\documentclass[3p]{elsarticle}
\usepackage{amsmath,amssymb,amsfonts}%
\usepackage{amsthm}%
\usepackage{mathrsfs}%
\usepackage{mathtools}

\usepackage{xcolor}%
\usepackage{lmodern}
\usepackage{hyperref}
\usepackage[nameinlink]{cleveref}
\usepackage{orcidlink}

\newcommand{\e}{\mathrm{e}}

\newcommand{\R}{\mathbb R}

\newcommand{\eps}{\varepsilon}

\renewcommand{\epsilon}{\varepsilon}

\crefname{ineq}{Inequality}{Inequalities}
\creflabelformat{ineq}{#2{(#1)}#3}
\crefname{assumption}{Assumption}{Assumptions}

\crefname{remark}{Rem.}{Remarks}
\crefname{lemma}{Lem.}{Lemmata}
\crefname{proposition}{Prop.}{Propositions}
\crefname{theorem}{Thm.}{Thms.}
\crefname{assumption}{Asm.}{Assumptions}
\crefname{definition}{Defn.}{Definitions}
\crefname{example}{Ex.}{Examples}
\crefname{section}{Sec.}{Sections}
\crefname{equation}{Eq.}{Eqs.}
\crefname{figure}{Fig.}{Figures}

\usepackage{color}
\definecolor{mygreen}{RGB}{28,172,0} 
\definecolor{mylilas}{RGB}{170,55,241}
\definecolor{faublue}{RGB}{0,56,101}
\definecolor{fauorange}{RGB}{201,147,19}
\definecolor{faured}{RGB}{141,20,41}
\definecolor{faucyan}{RGB}{0,177,235}
\definecolor{faugreen}{RGB}{0,122,93}
\definecolor{fau-nat-green}{RGB}{0,122,93}
\definecolor{faugray}{RGB}{152,164,174}

\definecolor{matlabred}{RGB}{162, 20, 47}
\definecolor{matlabgreen}{RGB}{119, 173, 48}
\definecolor{matlabcyan}{RGB}{77, 191, 239}
\definecolor{matlabblue}{RGB}{0, 114, 190}
\definecolor{matlabyellow}{RGB}{238, 178, 32}
\definecolor{matlaborange}{RGB}{218, 83, 25}
\definecolor{matlabpurple}{RGB}{126, 47, 142}

\definecolor{matlabred1}{RGB}{121.5000, 43.5000, 82.7500}
\definecolor{matlabred2}{RGB}{81.0000, 67.0000, 118.5000}
\definecolor{matlabred3}{RGB}{40.5000, 90.5000, 154.2500}

\definecolor{fgreen1}{RGB}{67, 176, 42}
\definecolor{fgreen2}{RGB}{44.66, 133, 64}
\definecolor{fgreen3}{RGB}{22.33, 90, 86}
\definecolor{fgreen4}{RGB}{0, 47, 108}

\definecolor{forange1}{RGB}{255, 184, 28}
\definecolor{forange2}{RGB}{241, 142, 32.5}
\definecolor{forange3}{RGB}{227.5, 100, 37}
\definecolor{forange4}{RGB}{213.75, 58, 41.5}
\definecolor{forange5}{RGB}{200, 16, 46}

\newtheorem{theorem}{Theorem}
\newtheorem{assumption}[theorem]{Assumption}%

\newtheorem{remark}{Remark}%

\newtheorem{definition}{Definition}%

\begin{document}

\title{The obstacle problem for scalar conservation laws with nonlocal dynamics}

\author[1]{Paulo
Amorim\,\orcidlink{0000-0001-5590-7409
}}
\ead{paulo.amorim@fgv.br}

\author[2]{Alexander Keimer\,\orcidlink{0000-0003-3825-5853}}
\ead{alexander.keimer@uni-rostock.de}

\author[3,4]{Lukas Pflug\,\orcidlink{0000-0001-8001-5832}}
\ead{lukas.pflug@fau.de}

\author[3,4]{Jakob Rodestock\,\orcidlink{0009-0007-5527-1341}}
\ead{jakob.w.rodestock@fau.de}

\affiliation[1]{organization={Escola de Matemática Aplicada, Fundação Getulio Vargas, FGV- EMAp},addressline={Praia de Botafogo 190}, city={Rio de Janeiro}, postcode={22250-900}, country={Brazil}}

\affiliation[2]{organization={Institute of Mathematics, Universität Rostock}, addressline={Ulmenstraße 69}, city={Rostock}, postcode={18057}, country={Germany}}

\affiliation[3]{organization={Department Mathematik, Friedrich-Alexander-Universität Erlangen-Nürnberg (FAU)}, addressline={Cauerstraße 11}, city={Erlangen}, postcode={91058}, country={Germany}}

\affiliation[4]{organization={Competence Center of Scientific Computing, Friedrich-Alexander-Universität Erlangen-Nürnberg (FAU)}, addressline={Martensstraße 5a}, city={Erlangen}, postcode={91058}, country={Germany}}

\begin{abstract}In this article, we present a method to find a solution to a one-dimensional nonlocal conservation law that respects a space-dependent mapping, referred to as the obstacle. This is achieved by generalizing existing results for the local conservation law: We consider a relaxation of the velocity, that explicitly depends on the obstacle. We prove existence of solutions to the relaxed problem and show that, as the relaxation mapping converges to a Heaviside-type function, the corresponding solutions converge to a weak solution of a discontinuous nonlocal conservation law. Moreover, we can characterize the limiting flux in several cases.

The paper concludes with a numerical study that illustrates the aforementioned convergence.
\end{abstract}


\maketitle

\section{Introduction}
The obstacle problem addresses the question of whether a solution of a partial differential equation (or a related quantity, like the flux of the PDE) or of a variational problem can be constrained by a prescribed function $o:\R^n \rightarrow \R$ beforehand, called \emph{the obstacle}. Its applications range from crowd and traffic control~\cite{1rodrigues,Andreianov2016,1garavello,1andreianov2,1dymski,1garavello2,Bayen2022,Benyahia2017,Marcellini2020} to membranes~\cite{Rodriguez-Aros2016-em}, finance~\cite{Dumitrescu2014,Perninge2025} and other fields. \\
For elliptic and parabolic PDEs, the problem has been studied extensively~\cite{FernandezReal2020,2bogelein,Lions1967,Caffarelli1977} (a small, non-exhaustive list). In contrast, we focus in this contribution on hyperbolic conservation laws in one dimension. More precisely, for $T>0$, $U \in C^2(\R)$, we consider
\begin{equation}\label{eq:scl}
\partial_t q(t, x) + \partial_x \big( q(t, x)U(q(t, x))\big) = 0, \quad (t, x)\in (0, T) \times \R.
\end{equation}
Such equations (and related variants) play a huge role in continuum mechanics and arise, for instance, in traffic flow~\cite{Greenberg1959} or gas dynamics~\cite{Lions1994}. For hyperbolic conservation laws, the theory of obstacle problems has developed only more recently. To the best of our knowledge, foundational results for the one-dimensional case were established by Lévi~\cite{1levi}. Since then, the theory has been extended in several directions: The obstacle has been enforced either through a source term~\cite{1amorim,Berthelin2003,Berthelin2003split}, an abstract projection argument~\cite{Santambrogio2018} or approximation~\cite{1colombo}. Furthermore, estimates on solutions and conincidence set have been derived in~\cite{1rodrigues2}. \\
In particular, Amorim et al.~\cite{obstacle} proposed enforcing the obstacle by multiplying the flux \(q\mapsto q\cdot U(q)\) in~\cref{eq:scl} by a penalty function $V_\varepsilon(o-q)$ (cf.~\cref{fig:V}), 
\begin{figure}
    \centering
    \includegraphics[width=0.5\linewidth]{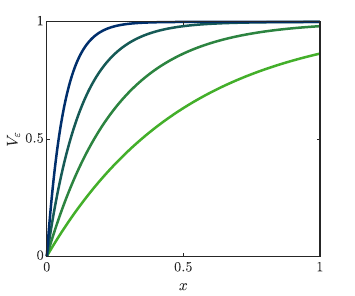}
    \caption{The mapping $V_\varepsilon(x) = 1 - \exp\left( - \tfrac{x}{\eps}\right)$ (where $x \in \R$) for $\eps = \textcolor{fgreen1}{2^{-1}}, \textcolor{fgreen2}{2^{-2}}, \textcolor{fgreen3}{2^{-3}}, \textcolor{fgreen4}{2^{-4}}$ relaxes the velocity when the obstacle is not adhered to.}
    \label{fig:V}
\end{figure}
which serves as a smooth approximation of the Heaviside function and is intended to ``slow down'' $q$ in front of the obstacle.\\
This approach has the advantage that it can be extended naturally to related models, such as the \emph{nonlocal scalar conservation law}
\begin{equation}
\begin{aligned} \notag   
\partial_t q(t, x) + \partial_x \left( q(t, x) U\big(W[q](t, x) \big)  \right) &= 0\quad  \text{ for } (t, x)\text{ in }(0, T) \times \R, \\
q(0, x) &= q_0(x) \quad \text{ for } x\text{ in } \R
\end{aligned}
\end{equation}
where \(q_{0}\in \mathrm{BV}(\R; \R_{\geq 0}),\) and $U:\R \rightarrow \R$ is a velocity function and 
\begin{equation}\label{eq:nlo}
    W[q](t, x) \coloneqq \int^{\infty}_x \gamma(y-x) q(t, y)~\mathrm{d}y\quad  \text{ for all }(t, x) \in (0, T) \times \R
\end{equation} and a kernel $\gamma \in L^1\big(\R_{\geq 0};\R_{\geq0}\big)$.\\
~\cref{eq:nocl} has the property of preserving the regularity of the initial datum $q_0$ and admitting a unique solution, even without an entropy condition~\cite{keimernonlocalbalance2017,Keimer2018bounded,Amorim2015nonlocalnumerical,Coclite2022BV,Kloeden2016,keimer2026bothsides}. This is not true for the local equation~\eqref{eq:scl}~\cite{Kruzkov1970,evanspartial}. The local and nonlocal models are related through the singular limit~\cite{Coclite2022SL,Amorim2015nonlocalnumerical,Crippa2021,Keimer2023disc,colombo2023overview,Colombo2023,Keimer2025,Colombo2019,Friedrich2024,dn25}. The nonlocal equation is a refined model for phenomena like traffic modeling, since the integral operator accounts for downstream density effects~\cite{Chiarello2018}. We will not examine, however, the role of the singular limit for the obstacle problem in this article.

The main objective of this paper is to investigate whether the penalization strategy introduced by Amorim et al.~\cite{obstacle} can be extended to the nonlocal equation. This results in a nonlinear-nonlocal partial differential equation~\cite{Chiarello2018,Chiarello2019stability,decourcel2025}. Very recently, existence and uniqueness of entropy solutions to this equation have been shown in a very general setting~\cite{keimerliverani26} using fixed point methods. However, we must use a different approach to prove existence, as we need specific estimates to account for the parameter $\varepsilon$ in $V_\varepsilon$.

We show that the arguments from~\cite{obstacle} can indeed be adapted to~\cref{eq:nocl}, provided that $U$ is decreasing and positive in the range of the obstacle, and $\gamma$ is decreasing and satisfies mild regularity assumptions.

More precisely, we prove that solutions $q_\varepsilon$ to the relaxed problem
\begin{align}
 \partial_t q_\varepsilon + \partial_x \left( V_\varepsilon\big(o(x)-q_\varepsilon\big) U\big(W[q_\varepsilon](t, x) q_\varepsilon \big)  \right)&= 0, &&(t, x)\in (0, T) \times \R, \label{eq:nocl}\\
  q(0, x)  &= q_0(x), && x \in \R, \label{eq:nocl1}
\end{align}
with the nonlocal operator as in \cref{eq:nlo} converge for an initial datum $q_0 \in \mathrm{BV}(\R)$ (along a subsequence) as $\varepsilon \searrow 0$ to a function $q \in C\big([0, T]; L^1_{loc}(\R)\big)$ that (still) obeys the obstacle, conserves mass, admits finite speed of propagation, and formally solves the discontinuous nonlocal conservation law
\begin{align*}
    \partial_t q(t,x)
    + \partial_x\big(
        \tilde{V}[q](t,x)\, q(t,x)\, U\bigl(W[q](t,x)\bigr)
      \big)
    &= 0,
    && (t,x) \in (0,T)\times \R, \\
    q(0,x)
    &= q_0(x),
    && x \in \R,
\end{align*}
where \(\tilde{V}[q](t,x)\in [0,1]\) for \((t,x)\in (0,T)\times\R\text{ a.e.}\) is a given \(L^{\infty}\) function. This formulation is in greater depth and rigor described in \cite{BULEK2011,Bulek2017}.
We also characterize explicitly the velocity profile in the case that the obstacle is active (otherwise, the velocity is just the nonlocal velocity \(U(W[q])\) given). The paper concludes with numerical simulations illustrating the qualitative behavior of the solutions and the effect of the obstacle.

\section{Preliminaries}
\noindent For $\eps\in \R_{>0}$, we consider the nonlinear-nonlocal problem~\cref{eq:nocl,eq:nocl1}
in $(0, T) \times \R$ with \textbf{initial condition} $q(0, \cdot) = q_0$. Here, $W$ is as in~\cref{eq:nlo}. We want to specify the other variables in
\begin{assumption}[Regularity of data]\label{as:as1}
    We assume that
    \begin{enumerate}
        \item[(U)]  the \textbf{velocity field} satisfies $U \in C^2\big(\R; \R\big)$, $U' \leq 0$ and there exists $U_{\min} \in \R_{>0}$ such that $U(x) \geq U_{\min}$ for all $x \in [-\lVert o \rVert_{L^\infty(\R)}-1, \lVert o \rVert_{L^\infty(\R)}+1]$,
       
        \item[(K)] the \textbf{kernel} $\gamma:\R_{\geq 0} \rightarrow \R$ is decreasing a.e., fulfills $\int_0^\infty \gamma(x)~\mathrm{d}x = 1$ and is of regularity $W^{2, 1}(\R_{\geq 0})$,

        \item[(O)] The \textbf{obstacle}~$o$ satisfies $o \in W^{2, \infty}(\R; \R_{>0})$, $o' \in W^{1, 1}(\R)$, $o_{\min} \coloneqq \inf_{y\in \R}o(y) \in \R_{>0}$
        
        and $\lim_{x \rightarrow \pm \infty}o(x)$ exist in $\R_{>0}$,

            \item[(I)] the \textbf{initial datum}~$q_0$ satisfies $q_0 \in \mathrm{BV}(\R; \R_{\geq 0})$.

        \item[(D)] and $o$ and $q_0$ are supposed to satisfy $d_o \coloneqq \underset{x \in \R}{\mathrm{essinf}}~ o(x)-q_0(x) > 0$.
    \end{enumerate}
    Furthermore, as a special case of~\cite[Assumption 1]{obstacle}, we set the \textbf{velocity relaxation factor} which is introduced in \cref{eq:nocl} to
    \begin{equation}\notag
        V_\varepsilon(x) \coloneqq 1-\exp\left(-\tfrac{x}{\varepsilon}\right)
    \end{equation}
    for all $x \in \R$. 
\end{assumption}

\begin{remark}[Justification of~\cref{as:as1} and implications]
\begin{enumerate}~
    \item[(i)] In the subsections, most notably in~\cref{thm:cp,thm:oslq,Thm:OSLV}, Oleinik-type estimates for compactness in $\eps$ that involve the quantities that were introduced in~\cref{as:as1} will be made that are very similar to what has been done~\cite{obstacle}. During these estimates, derivatives of $U$, $\gamma$ and $o$ up to order 2 will occur. This explains the high regularity and integrability assumptions. 
    \item[(ii)] For the reasons described in (i) (i.e., Oleinik-type estimates \cite{oleinik}), we need strict sign conditions on $U$ and $o$.
    \item[(iii)] The $\mathrm{BV}$-regularity is required to ensure that the viscosity approximation converges in a suitable topology~\cite{godlewski1991}.
    \item[(iv)] The condition (D) is important to ensure that the solution $q_{\nu,\eps}$ to~\cref{eq:nocl,eq:nocl1} strictly respects the obstacle for all times.

    \item[(v)] It holds that $\lim_{\eps\searrow 0} V_\eps(x) = 1$ for all $x \in \R_{>0}$ and $\lim_{\eps\searrow 0} V_\eps(0) = 0$. This means that $V_\eps$ is a smooth, monotone approximation of $\chi_{\R_{>0}}$ on $\R_{\geq 0}$.\\
    \item[(vi)] It holds that $\sup_{\eps > 0} \big\lVert V_\eps \big\rVert_{L^\infty(\R)} = 1$ and $\sup_{\eps > 0} \big \lVert V_\eps^{(k)} \big \rVert_{L^\infty(\R)} = \infty$ for $k \in \mathbb{N}$.
    \item[(vii)] Because of the comparison principle~\cref{thm:cp}, one can replace~\cref{as:as1} (U) by the assumption
    \begin{enumerate}
        \item[(U1)] $U \in W^{2, \infty}(\R)$, $U' \leq 0$ on $[-\lVert o \rVert_{L^\infty(\R)}-1, \lVert o \rVert_{L^\infty(\R)}+1]$ and there exists $U_{\min}\in \R_{>0}$ such that $U(x) \geq U_{\min}$ for all $x \in \R$,
    \end{enumerate}
    as $U$ will only be evaluated on a compact interval.
\end{enumerate}
    
\end{remark}
\section{Viscosity approximation and comparison principles}
We can derive several properties of the viscosity approximation. The latter allows for easier computations due to its smoothness. It will enable us to show one sided Lipschitz bounds that are crucial when letting $\eps \searrow 0$.
\subsection{Existence and uniqueness for the viscosity approximation}
Due to the combination of nonlinear, explicitly space-dependent flux and nonlocal impact, deriving existence and uniqueness for the equation is especially delicate.  We employ a fixed point argument similar to the one outlined in~\cite[p.~56, Lemma 2.1]{godlewski1991} and use some results for parabolic equations from the same book~\cite{godlewski1991}. First, we define:
\begin{definition}[Parabolic space]
    We set 
    \[
    \mathcal{W}(0, T) \coloneqq \left \lbrace v \in L^2\big((0, T); H^1(\R)\big) : \partial_t v \in L^2\big((0, T); L^2(\R) \big) \right \rbrace
    \]
   Remember that $\mathcal{W}(0, T) \hookrightarrow C\big([0, T]; L^2(\R)\big)$~\cite[p.54]{godlewski1991}.
\end{definition}
For this subsection, we set:
\begin{definition}[Viscosity approximation of~\cref{eq:nocl}]\label{def:va}
    Let~\cref{as:as1} hold and $\varphi_\nu$ be a standard mollifier~\cite[4.2.1 Notation (iii)]{evans}. A solution $q_{\nu, \eps} \in \mathcal{W}(0, T)$ in the sense of~\cite[p.~55]{godlewski1991} to the initial value problem for \((t,x)\in \in(0,T)\times\R\)
    \begin{align}
    \partial_t q_{\nu,\varepsilon}
    + \partial_x\Big(
        V_\varepsilon
            \big(o_\nu(x)-q_{\nu,\varepsilon}\big)
            q_{\nu,\varepsilon}
            U_\nu\bigl(W[q_{\nu,\varepsilon}](t,x)\bigr)
    \Big)
    \label{eq:vis1} 
    &=
    \nu\big(
        \partial_{xx}^2 q_{\nu,\varepsilon}
        - o_\nu''(x)
    \big)\\
    q_{\nu,\varepsilon}(0,\cdot)
    &\equiv q_{0,\nu}\coloneqq
    \bigl(q_0 * \varphi_\nu\bigr)\quad \text{on }\R
    \label{eq:vis3}
\end{align}
     is called \emph{viscosity approximation} of~\cref{eq:nocl,eq:nocl1}. Here,  $o_\nu \coloneqq o \ast \varphi_\nu$, $U_\nu \coloneqq \big( U\cdot \chi_{[-1, 1+\lVert o \rVert_{L^\infty(\R)}]}\big)*\varphi_\nu$ and, $W$ as in~\cref{eq:nlo}.\\
     The symbol ``$*$'' denotes a convolution in the sense of~\cite[4.13 (2)]{alt2016eng}.
\end{definition}

A weighted version of the viscosity approximation satisfies a similar parabolic partial differential equation:
\begin{remark}[Substitution]\label{rem:subs}
    Let $\nu, \eps \in \R_{>0}$ and let~\cref{as:as1} hold. If $q_{\nu, \eps} \in \mathcal{W}(0, T)$ solves~\cref{eq:vis1}, then the mapping $r_{\nu, \eps} \in \mathcal{W}(0, T)$ 
    \[
    r_{\nu, \eps}[\lambda](t, x) \coloneqq \exp(-\lambda t) q_{\nu, \eps}(t, x)
    \]
    for all $(t, x) \in (0, T) \times \R$, $\lambda \in \R_{>0}$ fulfills
    \begin{equation}
    \begin{aligned}
    \partial_t r_{\nu, \eps}[\lambda](t, x) &= - \lambda r_{\nu, \eps}[\lambda](t, x)   -  \partial_x\big(V_\eps(o_\nu(x)- \exp(\lambda t)r_{\nu, \eps}[\lambda](t, x))r_{\nu, \eps}[\lambda](t, x)\\
    & \qquad \cdot U_\nu\big(W[\exp(\lambda t)r_{\nu, \eps}[\lambda]](t, x) \big)\big)  + \nu\big(\partial_{xx}^2 r_{\nu, \eps}[\lambda](t, x) - \exp(-\lambda t) o_\nu''(x)\big) 
    \end{aligned}
     \label{eq:subs1} 
    \end{equation}
    in $(0, T) \times \R$ in a weak sense.
\end{remark}
This can be proved by plugging the formula into the weak formulation. This allows us to switch conveniently~\cref{eq:vis1} and~\cref{eq:subs1} for arbitrary $\lambda$. Now note:
\begin{theorem}[Solution operator for viscous problem is a contraction]\label{Thm:co}
    Let~\cref{as:as1} hold, $\nu, \eps, \lambda \in \R_{>0}$ be given and $F:L^2\big((0, T); L^2(\R)\big)$ $\rightarrow$ $L^2\big((0, T); L^2(\R)\big)$ be the operator that maps any $\varrho \in L^2\big((0, T); L^2(\R)\big)$ onto the unique weak solution $F[\varrho] \in \mathcal{W}(0, T)$ to the equation on \((0,T)\times\R\)
    \begin{align*}
     \partial_t r_{\nu, \eps}[\lambda]&=- \lambda r_{\nu, \eps}[\lambda] - \partial_x\Big(V_\eps(o_\nu(x)- \exp(\lambda t)r_{\nu, \eps}[\lambda])r_{\nu, \eps}[\lambda]\notag\\
    & \qquad\qquad\qquad\quad \cdot U_\nu\big(W[\exp(\lambda t)\varrho](t, x) \big)\Big) + \nu\big(\partial_{xx}^2 r_{\nu, \eps}[\lambda]- \exp(-\lambda t) o_\nu''(x)\big),\\
    r_{\nu, \eps}(0, \cdot) &\equiv q_{0, \nu},\qquad \text{ on } \R.
    \end{align*}
    Then, $F$ is well defined and a contraction for any finite \(T\in\R_{>0}\).
\end{theorem}
\begin{proof}
        First, we know that
        \[
        \mathrm{supp}(U_\nu) \subseteq \big[ -1-\nu, \lVert o \rVert_{L^\infty(\R)}+1+\nu \big]
        \]
        by definition of $U_\nu$ (cf.~\cref{def:va}). The maximum principle will later justify the choice of the support. 
        \\ 
        We observe that $h_g: \R \times \R \rightarrow \R$,
        \[
        h_g(x, u) \coloneqq V_\eps(g(x)-u\big)u\quad \forall (x, u) \in \R \times \R
        \]
        defines a bounded mapping that is smooth in $u$ for every $g\in L^\infty(\R)$. Moreover,
        \begin{align}
            \sup_{(x, u) \in \R \times \R} \left \lvert \partial_u h_g(x, u) \right \rvert \leq \lVert V_\eps \rVert_{L^\infty(\R)} +  \sup_{(t, x, u) \in (0, T) \times \R \times \R} \left \lvert V_\eps'(g(x)-u)u \right \rvert < \infty.
        \end{align}
         An argument outlined in~\cite[Theorem 4]{obstacle} can be used to conclude that the problem~\cref{eq:subs1} admits a unique solution in $\mathcal{W}(0, T)$, which means that $F$ is well-defined.
         
         To prove that $F$ is a contraction however, we argue like in~\cite[Lemma 2.1., p.~56]{godlewski1991}: Take $\varrho, \hat{\varrho} \in L^2\big((0, T); L^2(\R)\big)$, set $r \coloneqq F[\varrho]$, $\hat{r} \coloneqq F[\hat{\varrho}]$ and take a look at ($t \in (0, T)$):
        \begin{align*}
            & \left \langle \partial_t \big( r-\hat{r} \big)(t, \cdot), \big(r-\hat{r}\big)(t, \cdot) \right \rangle_{H^{-1}(\R)\times  H^1(\R)}.
            \intertext{We plug in the weak formulation of~\cref{eq:subs1} and using Young's convolution inequality~\cite[(4-24)]{alt2016eng} to reach after some computation:}
            & \leq -\lambda \int_{\R} \lvert r(t, x) - \hat{r}(t, x) \rvert^2~\mathrm{d}x\\
            & \quad + \tfrac{1}{4\nu} \lVert g_o \rVert^2_{L^\infty(\R)} \lVert U_\nu' \rVert_{L^\infty(\R)}^2 \lVert \gamma \rVert_{L^1(\R_{\geq 0})} \int_{\R} \lvert \varrho(t, x) - \hat{\varrho}(t, x) \rvert^2~\mathrm{d}x \\
            & \quad + \tfrac{1}{4\nu}  \lVert g_o' \rVert_{L^\infty(\R)}^2 \lVert U_\nu \rVert_{L^\infty(\R)}^2 \int_{\R} \lvert r(t, x) - \hat{r}(t, x) \rvert^2~\mathrm{d}x.
        \end{align*}
        Integrating and rearranging yields
        \[
        \lVert r - \hat{r}\rVert_{L^2((0, T); L^2(\R))} \leq \tfrac{\frac{1}{4\nu} \lVert g_o \rVert^2_{L^\infty(\R)} \lVert U_\nu' \rVert_{L^\infty(\R)}^2 \lVert \gamma \rVert_{L^1(\R_{\geq 0})}}{\lambda - \frac{1}{4\nu}  \lVert g_o' \rVert_{L^\infty(\R)}^2 \lVert U_\nu \rVert_{L^\infty(\R)}^2}\cdot  \lVert \varrho - \hat{\varrho} \rVert^2_{L^2((0, T); L^2(\R))}.
        \]
        Hence, $F$ is a contraction for sufficiently large $\lambda$. 
    \end{proof}

\begin{theorem}[Existence and uniqueness for~\crefrange{eq:vis1}{eq:vis3}]
    For $(\nu, \eps) \in \R_{>0}^2$, under~\cref{as:as1}, there is a unique solution $q_{\nu, \eps} \in H^5\big((0, T) \times \R\big) \hookrightarrow C^3\big([0, T] \times \R\big)$ to~\cref{eq:vis1}.
\end{theorem}
\begin{proof}
    \cref{Thm:co} allows for the application of Banach's fixed point Theorem~\cite[9.23 Theorem]{rudin1976} to the equation~\cref{eq:subs1}, which means that it has a unique solution. By~\cref{rem:subs}, we get the unique solution for the original equation.
\end{proof}

Like in~\cite{obstacle}, we make a remark:
\begin{remark}[Decay of $q_{\nu, \eps}$ at $\infty$]\label{rem:dec}
    Let~\cref{as:as1} hold and $q_{\nu, \eps}$ be a solution to~\cref{eq:vis1}. Due to the $H^5\big((0, T) \times \R\big)$-regularity, we can conclude that
    \[
    \lim_{\lvert x \rvert \rightarrow \infty} \partial_\alpha q_{\nu, \eps}(t, x) = 0,
    \]
    for all $t \in [0, T]$ and every multi-index $\alpha \in \mathbb{N}^2$ of order up to four~\cite[Corollary 8.9]{brezis2010functional}.
\end{remark}
\subsection{Comparison principles}

The viscosity approximation allows for the desired maximum principle $q_{\nu, \eps} < o$. We will be able to extend this to the result $q_\eps < o$. This shows that the velocity never vanishes completely, implying that no permanent ``traffic congestion'' takes place.
\begin{theorem}[Comparison principle for $q_{\nu, \eps}$]\label{thm:cp}
    Let~\cref{as:as1} hold and let $q_{\nu, \eps}$ be the solution to~\cref{eq:vis1}. We set $I \coloneqq \left[\frac{1}{2} o_{\min}, \lVert o \rVert_{L^\infty(\R)} + 1 \right)$. Then, there is a mapping $f_\eps:\R \rightarrow \R$ independent of $\nu$ such that
    \[
    -\tfrac{\nu}{ \lVert o' \rVert_{L^\infty(\R)} \lVert U \rVert_{L^\infty(I)}} \leq q_{\nu, \eps}(t, x) \leq f_\eps(x) < o(x)
    \]
    holds for all $(t, x) \in (0, T) \times \R$,  
    $\eps \in \R_{>0}$ and $\nu \in \left(0, 
    \lVert o' \rVert_{L^\infty(\R)} \lVert U \rVert_{L^\infty(I)}\right)$.
\end{theorem}
\begin{proof}
    Let $m:[0, T] \rightarrow \R$, $m_{\nu, \eps}(t) \coloneqq \sup_{x \in \R} q_{\nu, \eps}(t, x) - o(x)$ for $t \in [0, T]$. Due to $q_{\nu, \eps} \in H^5\big((0, T) \times \R \big)$, it is bounded and Lipschitz continuous. Hence, $m$ is well defined and differentiable a.e.~\cite[Theorem 3.2]{evans}.\\
    Now, fix any $t \in [0, T]$ where $m_{\nu, \eps}$ is differentiable. Then,
    \begin{equation}\notag
    m_{\nu, \eps}'(t) = \partial_t q_{\nu, \eps}\big(t, x^\ast(t)\big),
    \end{equation}
    where $x^\ast(t) \in \overline{\R}$ is chosen such that $q_{\nu, \eps}\big(t, x^\ast(t)\big) - o\big(x^\ast(t) \big) = \sup_{x \in \R} q_{\nu, \eps}(t, x) - o(x)$~\cite[Theorem 1]{Milgrom2002}. Here an evaluation at $-\infty$ or $\infty$ is understood in the limiting sense. So, owing to~\cref{rem:dec}, this is either $0$ or, if $x^\ast(t) \in \R$, we can plug in the PDE in~\cref{eq:vis1}. One can then compute, exploiting that $\partial_x q_{\nu, \eps}\big(t, x^\ast(t)\big) -o'\big(x^\ast(t)\big)= 0$ and $\partial_x^2 q_{\nu, \eps}\big( t, x^\ast(t)\big)- o''\big(x^\ast(t)\big) \leq 0$: 
    \begin{align}
        m_{\nu, \eps}'(t) & = - \partial_x \big( V_\varepsilon \big(o_\nu(x^\ast(t)) - q_{\nu, \eps}(t, x^\ast(t))\big)q_{\nu, \eps}(t, x^\ast(t)) U_\nu\big(W[q_{\nu, \eps}](t, x)\big) \big)  \notag\\
        & \quad + \nu \big(\partial_{xx}^2q\big(t, x^\ast(t)\big) -o''\big(x^\ast(t)\big) \big) \notag\\
        & = - V_\varepsilon'\big(o_\nu(x^\ast(t)) - q_{\nu, \eps}(t, x^\ast(t))\big)\big(o_\nu'(x^\ast(t)) - \partial_x q_{\nu, \eps}(t, x^\ast(t)) \big)\notag\\
        & \qquad \cdot q_{\nu, \eps}(t, x^\ast(t)) U_\nu\big(W[q_{\nu, \eps}](t, x)\big)\notag\\
        & \quad -  V_\varepsilon\big(o_\nu(x^\ast(t)) - q_{\nu, \eps}(t, x^\ast(t))\big)\partial_x q_{\nu, \eps}(t, x^\ast(t)) U_\nu\big(W[q_{\nu, \eps}](t, x)\big)\notag \\
        & \quad - V_\varepsilon\big(o_\nu(x^\ast(t)) - q_{\nu, \eps}(t, x^\ast(t))\big)q_{\nu, \eps}(t, x^\ast(t)) U_\nu'\big(W[q_{\nu, \eps}](t, x)\big)\notag \\
        & \qquad \cdot \partial_x W[q_{\nu, \eps}](t, x)+ \nu \big(\partial_{xx}^2q\big(t, x^\ast(t)\big) -o''\big(x^\ast(t)\big) \big).\notag
        \intertext{We plug in $o_\nu'(x^\ast(t)) - \partial_x q_{\nu, \eps}(t, x^\ast(t)) = 0$, wherever $\partial_x q_{\nu, \eps}$ appears and use $\partial_x^2 q_{\nu, \eps}\big( t, x^\ast(t)\big)- o''\big(x^\ast(t)\big) \leq 0$ (necessary optimality conditions)}
        & \leq V_\varepsilon\big(-m_{\nu, \eps}(t)\big)o_\nu'\big(x^*(t)\big) U_\nu\big(W[q_{\nu, \eps}](t, x)\big) \notag\\
        & \quad - V_\varepsilon\big(o_\nu(x^\ast(t)) - q_{\nu, \eps}(t, x^\ast(t))\big)q_{\nu, \eps}(t, x^\ast(t)) \lVert U_\nu ' \rVert_{L^\infty(\R)} \partial_x W[q_{\nu, \eps}](t, x). \notag\\
        \intertext{Applying the estimate $\lvert \partial_x W[q_{\nu, \eps}](t, x) \rvert \leq \lVert q_{\nu, \eps} \rVert_{L^\infty(\R)} \lvert \gamma \rvert_{\mathrm{TV}(\R_{\geq 0})} $ and estimating other terms, we obtain}
        & \leq  V_\varepsilon\big(-m_{\nu, \eps}(t)\big)\lVert o' \rVert_{L^\infty(\R)} \lVert U_\nu \rVert_{L^\infty(\R)} - V_\varepsilon\big(-m_{\nu, \varepsilon}(t)\big)\lVert q_{\nu, \eps} \rVert_{L^\infty(\R)}^2  \lVert U_\nu' \rVert_{L^\infty(\R)} \lvert \gamma \rvert_{\mathrm{TV}(\R_{\geq 0})}. \label{eq:ub1}
    \end{align} 
 According to the ODE comparison principle~\cite[Lemma 1.1]{teschlbook}, $m_{\nu, \eps} \leq g_{\nu, \eps}$, where $g_{\nu, \eps}:[0, T] \rightarrow\R$ solves the initial value problem
    \begin{align*}
    g_{\nu, \eps}'(t) & = V_\varepsilon(-g_{\nu, \eps}(t))\lVert o' \rVert_{L^\infty(\R)} \lVert U_\nu \rVert_{L^\infty(\R)} - V_\varepsilon\big(-g_{\nu, \varepsilon}(t)\big)\lVert q_{\nu, \eps} \rVert_{L^\infty(\R)}^2  \lVert U_\nu' \rVert_{L^\infty(\R)} \lvert \gamma \rvert_{\mathrm{TV}(\R_{\geq 0})}  \\
    g_{\nu, \eps}(0) & = \sup_{x \in \R} ~q_{0, \nu}(x) -o_\nu(x) \underset{\cref{as:as1}\text{(D)}}{<} 0
    \end{align*}
    on $[0, T]$. Now, note that $0$ is an equilibrium of the above ODE. So, since $g_{\nu, \eps}(0) < 0$, $g_{\nu, \eps} <0$ and hence $m_{\nu, \eps} <0$. We conclude that $q_{\nu, \eps}(t, x)< o(x)$ for all $(t, x) \in [0, T] \times \R$.\\
    Repeating the argument and applying the new bound $\lVert o \rVert_{L^\infty(\R)}$ to $\lVert q_{\nu, \eps} \rVert_{L^\infty(\R)}$ in~\cref{eq:ub1}, we can conclude that $f_\eps$ can be found. This is because the resulting initial value problem for the upper bound on $m_{\nu, \eps}$ will not depend on $\nu$ anymore.\\
    An analogous argument (cf. e.g.~\cite{obstacle}) applies for the lower bound.
\end{proof}

\section{Compactness of the family \texorpdfstring{$q_{\nu, \eps}$}{q} and existence for the hyperbolic equation}
Let~\cref{as:as1} hold. We aim to show that the limits $\lim_{\nu \searrow 0} q_{\nu, \eps}$ and $\lim_{\eps \searrow 0} q_{\eps}$ exist in $L^1_{loc}(\R)$. The former limit will yield existence for~\cref{eq:nocl} and the latter will give us the ``solution'' to the obstacle problem. For this, we can prove one sided Lipschitz (OSL) bounds:
\begin{theorem}[OSL bounds for $q_{\nu, \eps}$]\label{thm:oslq}
    Let \cref{as:as1} hold as well as that $q_0$ is one sided Lipschitz bounded from below, i.e., \[\exists c_{q_0} \in \R\ \forall (x,y)\in\R^{2}, x\neq y:\ \tfrac{q_0(x) -q_0(y)}{x-y}\geq c_{q_0}.\] 
    Then, the solution $q_{\nu, \eps}$ of \cref{eq:vis1} for \(\eps,\nu\in\R_{>0},\ \eps\leq \nu\) small enough satisfies
    \[
    \tfrac{q_{\nu, \eps}(t, x) - q_{\nu, \eps}(t, y)}{x-y}\geq \hat{K}
    \]
    for all \(t\in[0,T]\) and \((x,y)\in\R^{2},\ x\neq y\) with 
    \begin{align*}
        \hat{K}& = \hat{K}\big(\lVert o \rVert_{W^{2, \infty}(\R)}, \lVert \gamma \rVert_{W^{2, 1}(\R_{\geq 0})}, \lVert U \rVert_{W^{2, \infty}(J)}, \lVert o \rVert_{W^{2, \infty}(\R)}, c_{q_0} \big) \in \R\\
    J &\coloneqq \left[-\lVert o \rVert_{L^\infty(\R)} -1, \lVert o \rVert_{L^\infty(\R)} + 1\right] \subset\R.
    \end{align*}
\end{theorem}
\begin{proof}
    We can argue like in~\cite[Theorem 11]{obstacle} and/or in the proof of \cref{thm:cp} and estimate for \(t\in[0,T]\) $m_{\nu, \eps}(t) \coloneqq \inf_{y \in \R} \partial_y q_{\nu, \eps}(t, y)$. We know that 
    \[
    m_{\nu, \eps}'(t) = \partial_t \partial_2 q_{\nu, \eps}\big(t, x^*(t)\big) 
    \]
    where $x^*(t) \in \overline{\R}$ is chosen such that $m_{\nu, \eps}(t) = \partial_2 q_{\nu, \eps}\big(t, x^*(t)\big) - o\big(x^*(t)\big)$.  Moreover, let us assume that $\nu < \lVert o' \rVert_{L^\infty(\R)} \lVert U \rVert_{L^\infty(J)} \eps$ in order to get rid of the viscosity expression.
    
    Then, after further computation that is similar to the proof of \cref{thm:cp}, and applying the definition of $m_{\nu, \eps}$ and the optimality conditions $\partial_{xx} q_{\nu, \eps}(t, x) = 0$,  $o_\nu'''(x) - \partial_{xxx}q_{\nu, \eps}(t, x) \geq 0,$ when \(x\in\R\) is at time \(t\) one point where the minimum of \(\partial_{2}q_{\nu,\eps}\) is attained, one obtains the estimate when abbreviating with a slight abuse of notation \(V_{\eps}'\coloneqq V_{\eps}'(o_{\nu}-q_{\nu,\eps}\big)\) as well as recalling that \(V_{\eps}''\leqq 0\) and \(q_{\nu,\eps}\leqq o\) and \(\|V_{\eps}\|_{L^{\infty}(\R)}\leq 1\)
    \begin{equation}
    \begin{aligned}
        m_{\nu, \eps}'(t) & \geq V_\varepsilon'm_{\nu,\varepsilon}(t)^2 U_{\min} \\
        & \quad +V_{\eps}'m_{\nu, \eps}(t) \big(2\lVert o \rVert_{L^\infty(\R)}\lVert U \rVert_{L^\infty(\R)} + 2 \lVert o' \rVert_{L^\infty(\R)} \lVert U \rVert_{L^\infty(J)} +2 \lVert o \rVert_{L^\infty(\R)}^2 \lVert U' \rVert_{L^\infty(\R)} \lvert \gamma \rvert_{\mathrm{TV}(\R_{\geq 0})} \\
        & \qquad + 2 \lVert o' \rVert_{L^\infty(\R)} \lVert o \rVert_{L^\infty(\R)} \lVert U' \rVert_{L^\infty(J)} \lvert \gamma \rvert_{\mathrm{TV}(\R_{\geq 0})} + \lVert o'' \rVert_{L^\infty(\R)} \lVert o \rVert_{L^\infty(\R)} \lVert U \rVert_{L^\infty(J)}\\
        & \qquad - \lVert o \rVert_{L^\infty(\R)} ^2 \lVert U \rVert_{L^\infty(J)} -\lVert o'' \rVert_{L^\infty(\R)} \lVert o \rVert_{L^\infty(\R)} \lVert U \rVert_{L^\infty(J)} \big)\\
        &\quad- \lVert V_\varepsilon \rVert_{L^\infty(\R)} \lVert o \rVert_{L^\infty(\R)}^3 \lVert U'' \rVert_{L^\infty(\R)} \lvert \gamma \rvert_{\mathrm{TV}(\R_{\geq 0})}^2 
         - \lVert V_\varepsilon \rVert_{L^{\infty}(\R)} \lVert o \rVert_{L^\infty(\R)}^2 \lVert U' \rVert_{L^\infty(J)} \lVert \gamma \rVert_{W^{2, 1}(\R_{\geq 0})} \\
        & \quad + 2 \lVert V_\varepsilon \rVert_{L^\infty(\R)} \lVert U' \rVert_{L^\infty(J)} \lVert o \rVert_{L^\infty(\R)} \lvert \gamma \rvert_{\mathrm{TV}(\R_{\geq 0})} m_{\nu, \eps}(t). 
    \end{aligned}     
    \label{eq:oslqlast}
    \end{equation}
Setting
    \begin{align*}
        a_1 & \coloneqq U_{\min} \\
        a_2 & \coloneqq 2\lVert o \rVert_{L^\infty(\R)}\lVert U \rVert_{L^\infty(\R)} + 2 \lVert o' \rVert_{L^\infty(\R)} \lVert U \rVert_{L^\infty(J)}  +2 \lVert o \rVert_{L^\infty(\R)}^2 \lVert U' \rVert_{L^\infty(J)} \lvert \gamma \rvert_{\mathrm{TV}(\R_{\geq 0})} \\
        & \qquad + 2 \lVert o' \rVert_{L^\infty(\R)} \lVert o \rVert_{L^\infty(\R)} \lVert U' \rVert_{L^\infty(J)} \lvert \gamma \rvert_{\mathrm{TV}(\R_{\geq 0})}  + \lVert o'' \rVert_{L^\infty(\R)} \lVert o \rVert_{L^\infty(\R)} \lVert U \rVert_{L^\infty(J)} \\
        a_3 & \coloneqq -\lVert o \rVert_{L^\infty(\R)} \lVert U \rVert_{L^\infty(J)} -\lVert o'' \rVert_{L^\infty(\R)} \lVert o \rVert_{L^\infty(\R)} \lVert U \rVert_{L^\infty(J)},
    \end{align*}
    we study the behavior of  the polynomial $p(m_{\nu, \varepsilon}(t)) = a_{1}m_{\nu, \varepsilon}(t)^2 + a_{2}m_{\nu, \varepsilon} + a_3$ for very small $m_{\nu, \varepsilon}(t)$. We find for the zeros $a_{\pm} \in \R$: 
    \begin{align*}
        a_{\pm} = \tfrac{-a_{2} \pm\sqrt{a_{2}^2 - 4 a_{1}a_3}}{2a_{1}} .
    \end{align*}
    Since $a_1 > 0$, $p$ is positive, whenever $m_{\nu, \varepsilon}(t) < a_-$. So, either $m_{\nu, \varepsilon} \geq a_-$, or if $m_{\nu, \eps}(t) < a_-$, we can find, continuing from~\cref{eq:oslqlast} and estimating the bracket after $V_\eps'(o_\nu(x) -q_{\nu, \eps}(t, x))$ by $0$:
    \begin{align*} 
        m_{\nu, \eps}(t) \geq  & -  \lVert o \rVert_{L^\infty(\R)}^3 \lVert U'' \rVert_{L^\infty(J)} \lvert \gamma \rvert_{\mathrm{TV}(\R_{\geq 0})}^2  -  \lVert o \rVert_{L^\infty(\R)}^2 \lVert U' \rVert_{L^\infty(J)} \lVert \gamma \rVert_{W^{2, 1}(\R_{\geq 0})}  \\
        & \quad + 2 \lVert U' \rVert_{L^\infty(J)} \lVert o \rVert_{L^\infty(\R)} \lvert \gamma \rvert_{\mathrm{TV}(\R_{\geq 0})} m_{\nu, \eps}(t).
    \end{align*}
    Let us set (as stated in the theorem)
    \begin{align*}
        K_1 & \coloneqq -  \lVert o \rVert_{L^\infty(\R)}^3 \lVert U'' \rVert_{L^\infty(J)} \lvert \gamma \rvert_{\mathrm{TV}(\R_{\geq 0})}^2 
 -  \lVert o \rVert_{L^\infty(\R)}^2 \lVert U' \rVert_{L^\infty(J)} \lVert \gamma \rVert_{W^{2, 1}(\R_{\geq 0})}\\
        K_2 & \coloneqq 2  \lVert U' \rVert_{L^\infty(J)} \lVert o \rVert_{L^\infty(\R)} \lvert \gamma \rvert_{\mathrm{TV}(\R_{\geq 0})}.
    \end{align*}
   Applying the ODE comparison principle~\cite[Lemma 1.1]{teschlbook} (as in the proof of~\cref{thm:cp}), we obtain, taking into account all the different cases that we have considered so far: 
    \begin{align*}
        m_{\nu, \varepsilon}(t) & \geq \min\left \lbrace \exp\big(K_2 t\big)\left(c_{q_0} - \tfrac{K_2}{K_1} + \tfrac{K_2}{K_1}\exp(- K_2t) \right), c_{q_0}, a_- \right \rbrace \\
        & \geq \min\left \lbrace \exp\big(K_2 T\big)c_{q_0}  + \tfrac{K_2}{K_1} , c_{q_0}, a_- \right \rbrace
    \end{align*}
    for all $t \in (0, T)$, which concludes the proof, after taking into account the definition of $m_{\nu, \eps}$.  
\end{proof} 
Next, we establish some uniform \(\mathrm{BV}\) bounds on \(q_{\nu,\eps}\) uniform in \(\nu\) and \(\eps\):
\begin{theorem}[$\mathrm{BV}$ and $L^1$ bounds for $q_{\nu, \eps}$]\label{thm:bv}
    Under~\cref{as:as1}, the solutions $\big( q_{\nu, \eps}\big)_{\nu, \eps > 0}
    $ to~\cref{eq:vis1} satisfy:
    \begin{align*}
         \lVert q_{\nu, \eps} \rVert_{L^\infty((0, T); L^1(\R))} & \leq \nu T \lVert o'' \rVert_{L^1(\R)} + \lVert q_0 \rVert_{L^1(\R)} \\
        \lvert q_{\nu, \eps}(t, \cdot)\rvert_{\mathrm{TV}(\R_{\geq 0})} & \leq \exp\bigg(T\big( \lVert U  \rVert_{L^\infty(J)} \lVert o' \rVert_{L^\infty(\R)} + \lVert o \rVert_{L^\infty(\R)}^2 \lVert U' \rVert_{L^\infty(J)}\\
        & \qquad + \lVert V_\eps' \rVert_{L^\infty(\R)} \lVert o \rVert_{L^\infty(\R)}^2  \lVert U' \rVert_{L^\infty(J)} 2  \lVert \gamma \rVert_{W^{2, 1}(\R_{\geq 0})} \\
        & \qquad + 2 \lVert U' \rVert_{L^\infty(J)} \lVert o \rVert_{L^\infty(\R)}  \lVert \gamma \rVert_{W^{2, 1}(\R_{\geq 0})}\big)\bigg) \\
        & \qquad \cdot\Big(\lvert q_{0} \rvert_{\mathrm{TV}(\R_{\geq 0})}+T \lVert o' \rVert_{L^1(\R)} \lVert U \rVert_{L^\infty(J)} +2T\big(\nu T \lVert o'' \rVert_{L^1(\R)} + \lVert q_0 \rVert_{L^1(\R)} \big)\\
        & \qquad \qquad + T\big( 4 \lVert o' \rVert_{L^1(\R)} \lVert U' \rVert_{L^\infty(J)} \lVert o \rVert_{L^\infty(\R)} \lVert \gamma \rVert_{W^{2, 1}(\R_{\geq 0})} \\
        & \qquad \quad \quad  \quad \cdot \lVert U' \rVert_{L^\infty(J)} \lVert o \rVert^2_{L^\infty(\R)} \lVert \gamma \rVert^2_{W^{2, 1}(\R_{\geq 0})}  \big) \Big)
    \end{align*}
    for all $\nu, \eps \in \R_{>0}$, $t \in [0, T]$ when 
    \[
    J \coloneqq \left[-\lVert o \rVert_{L^\infty(\R)} -1, \lVert o \rVert_{L^\infty(\R)} + 1\right].
    \]
\end{theorem}
\begin{proof}
    Follow e.g.\ the steps in~\cite[p.~63, Theorem 2.3]{godlewski1991}.    
\end{proof}
Next, we look into the corresponding ``time compactness'':
\begin{theorem}[Time compactness of $q_{\nu, \eps}$]\label{thm:tc}
    Let~\cref{as:as1} hold. Then, the family $(q_{\nu, \eps})_{\nu, \eps > 0}$, i.e., the solutions to~\cref{eq:vis1}, satisfies
    \begin{align*}
      \left \lVert \partial_t q_{\nu, \eps}(t, \cdot) \right \rVert_{H^{-2}(\R)} \leq \big( \lVert U \rVert_{L^\infty(J)} +1 \big) \cdot \sqrt{\lVert o \rVert_{L^\infty(\R)} \big(\nu T \lVert o'' \rVert_{L^1(\R)} + \lVert q_0 \rVert_{L^1(\R)} \big)}
    \end{align*}
    for all $t \in (0, T)$, where  $J\coloneqq \left[-\lVert o \rVert_{L^\infty(\R)} -1, \lVert o \rVert_{L^\infty(\R)} + 1\right].$
\end{theorem}
\begin{proof}
We estimate the norm as follows for \(t\in[0,T]\)
    \begin{align*}
         \left \lVert  \partial_t q_{\nu, \eps}(t, \cdot)\right \rVert_{H^{-2}(\R)}& =\!\!\!\!\!\!\!\! \sup_{\substack{v \in H_0^{2}(\R) \\  \lVert v \rVert_{H^2(\R)} = 1}}\!\!\! \left \lvert \int_{\R}  \partial_t q_{\nu, \eps}(t, x)~\mathrm{d}x \right \rvert
         \intertext{apply Fubini's Theorem and plug in~\cref{eq:nocl},}
         & =\!\!\!\!\!\!\!\! \sup_{\substack{v \in H_0^{2}(\R) \\  \lVert v \rVert_{H^2(\R)} = 1}} \!\!\!\bigg \lvert  \!\int_{\R}\!\! \partial_x\big(V_\eps\big(o_\nu(x) -q_{\nu, \eps}(t,x)\big) q_{\nu, \eps}(t, x)U_\nu\big(W[q_{\nu, \eps}](t, x)\big) \big)v(x)\\
         &\qquad\qquad\qquad
         -\partial_{xx}q_{\nu, \eps}(t, x)v(x) ~\mathrm{d}x  \bigg \rvert,
         \intertext{integrate by parts}
         & = \!\!\!\!\!\!\!\!\sup_{\substack{v \in H_0^{2}(\R) \\  \lVert v \rVert_{H^2(\R)} = 1}}\!\!\!  \bigg \lvert\int_{\R} V_\eps\big(o_\nu(x) -q_{\nu, \eps}(t,x)\big) q_{\nu, \eps}(t, x)U_\nu\big(W[q_{\nu, \eps}](t, x) \big)v'(x)  - q_{\nu, \eps}(t, x)v''(x) ~\mathrm{d}x \bigg \rvert,
        \intertext{applying Hölder's inequality}
         & \leq \sup_{\substack{v \in H_0^{2}(\R) \\  \lVert v \rVert_{H^2(\R)} = 1}}  \lVert V_\eps \rVert_{L^\infty(\R)} \sqrt{\lVert q_{\nu, \eps} \rVert_{L^1(\R)} \lVert q_{\nu, \eps} \rVert_{L^\infty(\R)} } \lVert U \rVert_{L^\infty(J)} \lVert v' \rVert_{L^2(\R)} \\
         & \qquad\qquad\qquad + \sqrt{\lVert q_{\nu, \eps} \rVert_{L^1(\R)} \lVert q_{\nu, \eps} \rVert_{L^\infty(\R)} } \lVert v'' \rVert_{L^2(\R)},
         \intertext{and finally taking advantage of the estimates  in \cref{thm:bv} yields}
         &  \leq \sqrt{\lVert o \rVert_{L^\infty(\R)} \big(\nu T \lVert o'' \rVert_{L^1(\R)} + \lVert q_0 \rVert_{L^1(\R)}} \big)\big( \lVert U \rVert_{L^\infty(J)} +1 \big).
        \end{align*}

This is the claimed inequality.
\end{proof}

Finally, we obtain the main
\begin{theorem}[Existence for~\cref{eq:nocl} and compactness of \texorpdfstring{$q_{\nu, \eps}$}{q}]\label{Thm:exc}
    Let~\cref{as:as1} hold and $\big(q_{\nu, \eps}\big)_{\nu, \eps > 0}$ be the family of solutions to~\cref{eq:vis1,eq:vis3} for some $q_0 \in \mathrm{BV}(\R; \R_{\geq 0})$. Then, the following two statements are true:
    \begin{description}
        \item[a)] For $\eps \in \R_{>0}$, the limit $q_\eps \coloneqq \lim_{\nu \searrow 0} q_{\nu, \eps}$ exists in $C\big([0, T]; L^1_{loc}(\R) \big)$ and is the unique entropy solution to~\cref{eq:nocl}. Moreover, $0 \leq q_\eps < o$ holds a.e.
        \item[b)] If $q_0$ additionally fulfills the requirements of~\cref{thm:oslq}, then, along a subsequence, there exists $q \coloneqq \lim_{\eps \searrow 0} q_\eps$ in $C\big([0, T]; L^1_{loc}(\R)\big)$, which is the solution to a nonlocal-discontinuous PDE that is to be specified in~\cref{Thm:lim}. 
    \end{description}
\end{theorem}
\begin{proof}
    \cref{thm:oslq,thm:tc,thm:bv} give the desired compactness in $C\big([0, T]; L^2(\Omega) \big)$ for every compact interval $\Omega \subseteq \R$ if they are combined with~\cite[Theorem 6]{Simon1986CompactSI} and~\cite[Theorem 3.47]{Ambrosio2000} (for the embedding $\mathrm{BV}(\Omega) \hookrightarrow \hookrightarrow L^2(\Omega)$). A diagonal argument yields compactness in $C\big([0, T]; L^2_{loc}(\R)\big)$ and thus in $C\big([0, T]; L^1_{loc}(\R)\big)$. \\
    To prove that $q_\eps$ is actually an entropy solution, Kruzhkov's arguments~\cite{Kruzkov1970} can be imitated and the bounds $0 \leq q_\eps < o$ are obtained from~\cref{thm:cp}.
\end{proof}

We want to remind the reader that a) has already been shown for a more general equation under similar regularity assumptions~\cite{Amorim2015nonlocalnumerical}.\\

By simply applying the limit, we obtain our main result that defines the solution to the obstacle problem.
\begin{theorem}[Limiting problem as $\varepsilon \searrow 0$]\label{Thm:lim}
    Take the assumptions as in \cref{Thm:exc} and consider a subsequence such that $q_\eps \rightarrow q$ in $C\big([0, T];L^1_{\text{loc}}(\R)\big)$ and $V_\eps(o-q_\eps) \overset{\ast}{\rightharpoonup} \tilde{V}$ in $L^\infty\big((0, T) \times \R \big)$. Then, $q$ fulfills the Cauchy problem
    \begin{align*}
\partial_t q(t, x) + \partial_x \big(\tilde{V}(t, x) q(t, x) U\big(W[q](t, x)\big) \big) &= 0, && (t, x) \in (0, T) \times \R ,\\
q(0, x) &= q_0(x), && x \in \R,
\end{align*}
    in the sense defined in \cite{BULEK2011,Bulek2017}. Moreover, one can show $\tilde{V}(t, x) = 1$ a.e.\ \((t,x)\in (0,T)\times\R\) where $q(t, x) \neq o(x)$. 
\end{theorem}

\section{Characterization of \texorpdfstring{$\lim_{\eps \searrow 0} V_{\eps}$}{Vε}}
To get a better picture of the limit $\lim_{\eps\searrow 0} q_\eps$, one sided Lipschitz bounds of $V_\eps(o-q)$ are useful:
\begin{theorem}[OSL bounds for $V_\eps(o_\nu-q_{\nu, \eps})$]\label{Thm:OSLV}
    If assumption~\cref{as:as1} holds and $q_0$ is one sided Lipschitz continuous from below, then, for every $\eps \in \R_{>0}$, there is $\bar \nu(\eps) > 0$ and
    
    $\hat{C}= \hat{C}\big(U_{\min}, \lVert U\rVert_{W^{2, \infty}(J)}, \lVert \gamma \rVert_{W^{2, 1}(\R_{\geq 0})}, c_{q_0}, \lVert o \rVert_{W^{2, 1}(\R_{\geq 0})}  \big) \in \R$ such that
    \begin{align*}
    & \tfrac{V_\eps(o(x) -q_{\nu, \eps}(t, x)) - V_\eps(o(y)-q_{\nu, \eps}(t, y))}{x-y} \leq \hat{C} 
    \end{align*}
    for all $t \in (0, T)$ and $0 \leq \nu \leq \bar \nu(\eps)$. Here, $\hat{K}$ denotes the lower one sided Lipschitz bound from~\cref{thm:oslq} and $J \coloneqq \left[-\lVert o \rVert_{L^\infty(\R)} -1, \lVert o \rVert_{L^\infty(\R)} + 1\right]$.
\end{theorem}
\begin{proof}
    We set for $t \in (0, T)$.
    \begin{equation}
    \begin{aligned}
    m_{\nu, \eps}(t)& \coloneqq \sup_{y \in \R} \partial_y \big( V_\eps\big(o_\nu(y)-q_{\nu, \eps}(t, y)\big)\big)\\
    &=\sup_{y \in \R} V_\eps'\big(o_\nu(y)-q_{\nu, \eps}(t, y)\big)\big(o_\nu'(y) - \partial_y q_{\nu, \eps}(t, y)\big)
    \end{aligned}
    \label{eq:defi_mu_nu_eps}
    \end{equation}
It is sufficient to find an upper bound for it following the arguments of~\cref{thm:cp}. Set $t \in (0, T)$, let $x \in \R$ (depending on $t$) such that
    \begin{align*}
        m_{\nu, \eps}(t) = V_\eps'\big(o_\nu(x)-q_{\nu, \eps}(t, x)\big)(o_\nu'(x) - \partial_{2} q_{\nu, \eps}(t, x)),
    \end{align*}
    excluding the case $x \in \lbrace -\infty, \infty \rbrace$, where the right hand side vanishes anyway. 
    Further, note that optimality conditions for \(x\in\R\) being a point where the supremum is attained are:
    \begin{align*}
        0 & = V_\eps''\big(o_\nu(x)-q_{\nu, \eps}(t, x)\big)\big(o_\nu'(x) - \partial_2 q_{\nu, \eps}(t, x)\big)^2  + V_\eps'\big(o_\nu(x)-q_{\nu, \eps}(t, x)\big)(o_\nu''(x) - \partial_{2}^{2} q_{\nu, \eps}(t, x)) \\
        0 &\geq  V_\eps''' \big(o_\nu(x)-q_{\nu, \eps}(t, x)\big)(o_\nu'(x) - \partial_2 q_{\nu, \eps}(t, x))^3 \\
        & \quad + 3 V_\eps''(o_\nu(x)-q_{\nu, \eps}(t, x))(o_\nu'(x) - \partial_2 q_{\nu, \eps}(t, x))(o_\nu''(x) - \partial_{2}^{2} q_{\nu, \eps}(t, x)) \\
        & \quad + V_\eps'(o_\nu(x)-q_{\nu, \eps}(t, x))(o_\nu'''(x) - \partial_{2}^{3} q_{\nu, \eps}(t, x)).
    \end{align*}
    Abbreviating\( V_\eps^{(k)} \coloneqq V_\eps^{(k)}\big(o_\nu(x)-q_{\nu, \eps}(t, x)\big)\)    
    and leaving out the arguments of the other involved functions, i.e.,  \((t,x)\) or \(x\), we distinguish the two following cases.
\begin{description}
        \item[\textbf{Case 1:} ($q_{\nu, \eps} \leq \tfrac{o_\nu}{2}$)] In this case, the desired bound can be estimated directly (not using the abbreviated writing, with $(t, x)$ as in described above):
        \begin{align*}
             V_\eps'\big(o_\nu(x) -q_{\nu, \eps}(t, x) \big) (o_\nu'(x) - \partial_x q_{\nu, \eps}(t, x)) & = \tfrac{1}{\eps} \e^{-\tfrac{o_\nu(x) -q_{\nu, \eps}(t, x)}{\eps}}\big(o_\nu'(x) - \partial_x q_{\nu, \eps}(t, x)\big) \\
            & \leq \tfrac{1}{\eps} \exp\left(-\tfrac{o_{\min}}{2\eps} \right) ( o_\nu'(x) - \partial_x q_{\nu, \eps}(t, x)) \\
            \intertext{and using \cref{thm:oslq}, where $\hat{K} \in \R$ denotes the OSL bound for $\partial_x q$ gives}
            & \leq  \tfrac{1}{\eps} \exp\left(-\tfrac{o_{\min}}{2\eps} \right) \big( - \hat{K} + \lVert o' \rVert_{L^\infty(\R)} \big) \\
            & \leq \tfrac{2}{o_{\min}} \exp(-1) \big( - \hat{K} + \lVert o' \rVert_{L^\infty(\R)} \big),
        \end{align*}
        so that in this case, we indeed obtain the upper bound on the spatial derivative of \(V_{\eps}(o-q_{\nu,\eps})\), uniform in \(\eps,\nu\in\R_{>0}\).
        \item[\textbf{Case 2:} ($q_{\nu, \eps} > \tfrac{o_\nu}{2}$)] After some computations and estimates, we find:
        \begin{align*}
            m_{\nu, \eps} & \leq -2U_{\min}m_{\nu, \eps}^2+  m_{\nu, \eps}^2\big( -\tfrac{1}{2\eps}o_\nu \big)  \\
            & \quad + m_{\nu, \eps}\Big( \tfrac{1}{\eps}  \lVert U \rVert_{L^\infty(J)} + \tfrac{1}{\eps}  \lVert U \rVert_{L^\infty(J)} \lVert o' \rVert_{L^\infty(\R)} \\
            & \qquad + \tfrac{1}{\eps} \lVert V_\eps \rVert_{L^\infty(\R)} \lVert o \rVert_{L^\infty(\R)}^2 \lVert U' \rVert_{L^\infty(J)} \lvert \gamma\rvert_{\mathrm{TV}(\R_{\geq 0})} + \tfrac{2}{\eps} \lVert o' \rVert_{L^\infty(\R)} \lVert U \rVert_{L^\infty(J)} \\
            & \qquad + \lVert o \rVert_{L^\infty(\R)}^2 \lVert U' \rVert_{L^\infty(J)} \lvert \gamma \rvert_{\mathrm{TV}(\R_{\geq 0})} + \lvert \gamma \rvert_{\mathrm{TV}(\R_{\geq 0})} \lVert o \rVert_{L^\infty(\R)} \\
            & \qquad + \tfrac{1}{\eps} \lVert o \rVert_{L^\infty(\R)}^2\lVert U' \rVert_{L^\infty(J)} \lvert \gamma \rvert_{\mathrm{TV}(\R_{\geq 0})} -  \lVert U'' \rVert_{L^\infty(J)} \lvert \gamma \rvert_{\mathrm{TV}(\R_{\geq 0})}^2 \lVert o \rVert_{L^\infty(\R)}^2\\
            & \qquad -  \lVert U' \rVert_{L^\infty(J)} \lVert \gamma \rVert_{W^{2, 1}(\R_{\geq 0})} \lVert o \rVert_{L^\infty(\R)}  \Big) \\
            & \quad + \tfrac{2}{\eps} \lVert o'' \rVert_{L^\infty(\R)} \lVert U \rVert_{L^\infty(J)} + \tfrac{1}{\eps} \lVert o'' \rVert_{L^\infty(\R)}  \lVert U \rVert_{L^\infty(J)}  \\
            & \quad + \tfrac{2}{\eps}  \lVert o \rVert_{L^\infty(\R)}^2 \lVert U' \rVert_{L^\infty(J)} \lvert \gamma \rvert_{\mathrm{TV}(\R_{\geq 0})} + \tfrac{1}{\eps} \lVert o \rVert_{L^\infty(\R)}^3 \lVert U'' \rVert_{L^\infty(J)} \lvert \gamma \rvert_{\mathrm{TV}(\R_{\geq 0})} \\
            & \quad + \tfrac{1}{\eps}  \lVert o \rVert_{L^\infty(\R)}^2 \lVert U' \rVert_{L^\infty(J)}  \lVert \gamma \rVert_{W^{2, 1}(\R_{\geq 0})}  \\
            & \quad - \tfrac{\nu}{(V_\eps')^2 \eps} m_{\nu, \eps}^3 + \tfrac{ 2\nu }{\eps^2 (V_\eps')^2} m_{\nu, \eps}^3.
         \end{align*}
         Now, we set (recalling that \(\|V_{\eps}\|_{L^{\infty}(\R)}\leq 1\))
         \begin{align*}
             b_1 & \coloneqq -\tfrac{1}{2\eps}o_\nu \\
             b_2 & \coloneqq \tfrac{1}{\eps}  \lVert U \rVert_{L^\infty(J)} + \tfrac{1}{\eps}  \lVert U \rVert_{L^\infty(J)} \lVert o' \rVert_{L^\infty(\R)} \\
            & \qquad + \tfrac{1}{\eps} \lVert o \rVert_{L^\infty(\R)}^2 \lVert U' \rVert_{L^\infty(J)} \lvert \gamma\rvert_{\mathrm{TV}(\R_{\geq 0})} + \tfrac{2}{\eps} \lVert o' \rVert_{L^\infty(\R)} \lVert U \rVert_{L^\infty(J)} \\
            & \qquad + \lVert o \rVert_{L^\infty(\R)}^2 \lVert U' \rVert_{L^\infty(J)} \lvert \gamma \rvert_{\mathrm{TV}(\R_{\geq 0})} +  \lvert \gamma \rvert_{\mathrm{TV}(\R_{\geq 0})} \lVert o \rVert_{L^\infty(\R)} \\
            & \qquad + \tfrac{1}{\eps} \lVert o \rVert_{L^\infty(\R)}^2\lVert U' \rVert_{L^\infty(J)} \lvert \gamma \rvert_{\mathrm{TV}(\R_{\geq 0})} -  \lVert U'' \rVert_{L^\infty(J)} \lvert \gamma \rvert_{\mathrm{TV}(\R_{\geq 0})}^2 \lVert o \rVert_{L^\infty(\R)}^2\\
            & \qquad -  \lVert U' \rVert_{L^\infty(J)} \lVert \gamma \rVert_{W^{2, 1}(\R_{\geq 0})} \lVert o \rVert_{L^\infty(\R)} \\
            b_3 & \coloneqq  \tfrac{2}{\eps} \lVert o'' \rVert_{L^\infty(\R)} \lVert U \rVert_{L^\infty(J)} + \tfrac{1}{\eps}  \lVert o'' \rVert_{L^\infty(\R)}  \lVert U \rVert_{L^\infty(J)}  \\
            & \quad\quad + \tfrac{2}{\eps}  \lVert o \rVert_{L^\infty(\R)}^2 \lVert U' \rVert_{L^\infty(J)} \lvert \gamma \rvert_{\mathrm{TV}(\R_{\geq 0})} + \tfrac{1}{\eps} \lVert o \rVert_{L^\infty(\R)}^3 \lVert U'' \rVert_{L^\infty(J)} \lvert \gamma \rvert_{\mathrm{TV}(\R_{\geq 0})} \\
            & \quad  \quad + \tfrac{1}{\eps}\lVert o \rVert_{L^\infty(\R)}^2 \lVert U' \rVert_{L^\infty(J)}  \lVert \gamma \rVert_{W^{2, 1}(\R_{\geq 0})}  
         \end{align*}
         so that the previously derived inequality becomes
         \begin{align*}
         \partial_{t}m_{\nu, \eps} & \leq -2U_{\min}m_{\nu, \eps}^2+  b_{1}m_{\nu, \eps}^2+ m_{\nu, \eps}b_{2} +b_{3}- \tfrac{\nu}{(V_\eps')^2 \eps} m_{\nu, \eps}^3 + \tfrac{ 2\nu }{\eps^2 (V_\eps')^2} m_{\nu, \eps}^3.
        \end{align*}
         As \(b_1<0\) and of order \(\tfrac{1}{\eta}\) as is \(b_{2}\) it becomes evident that for large enough \(m_{\nu,\eps}\), we have
         \[
        b_{1}m_{\nu,\eps}^{2}+b_{2}m_{\nu,\eps}\leq 0
         \]
         Let $\hat{b}$ be the upper bound on \(m_{\nu,\eps}\) where this is true. Then, we obtain
         \begin{align*}
         m_{\nu, \eps}'(t) & \leq -U_{\min} m_{\nu, \eps}^2- \tfrac{\nu}{(V_\eps')^2 \eps} m_{\nu, \eps}^3 + \tfrac{2\nu}{\eps^2 (V_\eps')^2} m_{\nu, \eps}^3, 
         \intertext{and as the second summand on the right hand side is negative, we can further estimate}
         & \leq -U_{\min} m_{\nu, \eps}^2 + \tfrac{2\nu}{\eps^2 (V_\eps')^2} m_{\nu, \eps}^3.
         \intertext{Furthermore, we can recall the definition of \(m_{\nu,\eps}\) in \cref{eq:defi_mu_nu_eps} and plug it in to arrive at}
         & \leq -U_{\min} m_{\nu, \eps}^2 + \tfrac{2\nu}{\eps^2 }V_\eps' (o_\nu' - \partial_x q_{\nu, \eps})^3
         \intertext{and taking advantage of $V_\eps' \leq\tfrac{1}{\eps}$ and the aforementioned OSL bound $\hat{K}$ for $q_{\nu, \eps}$ from below as derived in \cref{thm:oslq} gives}
         & \leq -U_{\min} m_{\nu, \eps}^2 + \tfrac{2\nu}{\eps^3 } (o_\nu' - \partial_x q_{\nu, \eps})^3 \\
         & \leq -U_{\min} m_{\nu, \eps}^2 + \tfrac{2\nu}{\eps^3 } \big(\lVert o_\nu' \rVert_{L^\infty(\R)} - \hat{K}\big)^3
         \intertext{and assuming $\nu \leq \tfrac{\eps^3 \left(\lVert o_\nu' \rVert_{L^\infty(\R)} - \hat{K}\right)^3}{2}$ which is valid following \cref{thm:oslq}} 
         & \leq  -U_{\min} m_{\nu, \eps}^2 + 1. 
         \end{align*}
Recalling that \(U_{\min}>0\) according to \cref{as:as1}, it holds for $m_{\nu, \eps} > \tfrac{1}{\sqrt{U_{\min}}}$ that $m_{\nu, \eps}' \leq 0$, meaning \(\mu_{\nu,\eps}\) cannot grow beyond \(\tfrac{1}{\sqrt{U_{\min}}}\).
    \end{description}
    Combining \textbf{cases 1 and 2} yields altogether
   \begin{align*}
       m_{\nu, \eps}(t) \leq \max \left \lbrace \hat{b}, \tfrac{1}{\sqrt{U_{\min}}}, \tfrac{2}{o_{\min}} \exp(-1) \big( - \hat{K} + \lVert o' \rVert_{L^\infty(\R)} \big)\right \rbrace
       \forall t \in (0, T)
   \end{align*}
   concluding the proof.
   
\end{proof}

We can deduce the
\begin{remark}[Structure of $\lim_{\eps \searrow 0}V_\eps(o-q_{\eps})$]\label{rem:sv}
Let~\cref{as:as1} and the assumptions from~\cref{Thm:exc} hold, $q_\eps$ be the solution to~\cref{eq:nocl,eq:nocl1} for $\eps \in \R_{>0}$ and $q$ be from~\cref{Thm:lim}.\\
Now, similar to~\cite[Section 5]{obstacle}, we can at least \emph{formally} derive what is an implicit expression for $\tilde{V}$, which we assume to be the {weak\;$-^\ast$} limit in $L^\infty\big((0, T) \times \R\big)$ of $V_\eps(o-q_\eps)$ along the same subsequence where $\big(q_\eps\big)_{\eps > 0}$ convergences. \\
Fix $t \in (0, T)$. Moreover, formally let $E(t) \coloneqq \lbrace x \in  \R: q(t, x) = o(x)  \rbrace$ the coincidence set that is assumed to be \emph{connected}. Then, if all the involved expressions are regular enough, there exists $c \in \R$ such that
\[
\tilde{V}[q] (t, x) o(x) U\big(W[q](t, x) \big) = c,
\]
for a.e. $x \in E(t)$.\\
From this, using the OSL bounds on $v$, it can be derived that $\tilde{V}$ does not necessarily have to be zero, where $o = q$. For example, if $E (t) $ is an interval $[a, b]$ for some $t \in (0, T)$, then it can be shown that $c = o(b)V\big( W[q](t, b) \big)$---again, if data is sufficiently regular. This means that a ``full stopping effect'' can be avoided. The exact steps in this argument can be found in~\cite[Section 5]{obstacle}.
\end{remark}

\section{Numerical illustrations}
To validate our theoretical results, we conduct a numerical study. We choose $o:\R \rightarrow \R$, $o(x) \coloneqq -\exp\big(-(x-1)^2\big) + 1.2$ for all $x \in \R$ as the obstacle and study either the initial datum $q_0^1(x) \coloneqq \chi_{\left[-1.5, -1\right]}(x)-x\chi_{\left(-1, 0 \right]}(x)$  or $\chi_{[-1.5, -1]}(x)$ for all $x \in \R$ (cf.~\cite{obstacle} for both initial data). \emph{For the entire section}, the nonlocal kernel is chosen to be $\gamma:\R_{\geq 0} \rightarrow \R_{\geq0}$, $\gamma(x) \coloneqq \exp(-x)$ for all $x \in \R$. The velocity is either the LWR-type $U_1(x) = \tfrac{2}{3}\left( \tfrac{3}{2}-x\right)$ for $x \in \R$ or linear $U_2(x) = 1$ for $x \in \R$, as specified later.

Similar to~\cite{obstacle}, a Godunov scheme is employed. We can only use~\cref{eq:nocl} as an approximation, as we do not have an algorithm for the limiting problem found in~\cref{Thm:lim} at hand. Due to the convergence result for $\eps$,~\cref{eq:nocl} for $\eps$ is a suitable approximation. The nonlocality, however, has to be considered: The Godunov scheme~\cite[(32) and (34)]{Zhang2003} is the groundwork and the nonlocal operator is approximated using the scheme in~\cite{Friedrich2018}. The spatial step size is chosen to be $\Delta x = \tfrac{1}{5000}$. Boundary conditions are of homogeneous Neumann type.\\
\subsection{Numerical convergence as \texorpdfstring{$\eps \searrow 0$}{ε→0} for the nonlinear nonlocal equation}
First, we want to observe how the solution to~\cref{eq:nocl} behaves for different $\eps$, if the initial datum is chosen as $q_0^1$ (which is in accordance with~\cref{as:as1}) and the velocity is chosen to be $U_1$. In~\cref{fig:conv}, we can observe that convergence seems to occur and that the solution moves closer to the obstacle for smaller $\eps$. 
\begin{figure}
    \centering
    \includegraphics[clip,trim=0 0 7 5 0]{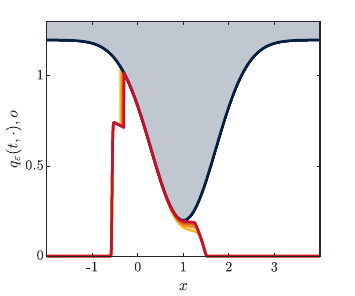}
    \includegraphics[clip,trim= 21 0 5 5 0]{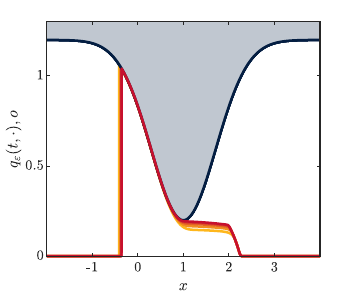}
    \caption{The solution $q_\eps$ to the nonlocal nonlinear equation~\cref{eq:nocl} with initial datum $q_0^1$ for $\eps = \textcolor{forange1}{2^{-6}}, \textcolor{forange2}{2^{-7}}, \textcolor{forange3}{2^{-8}}, \textcolor{forange4}{2^{-9}}, \textcolor{forange5}{2^{-10}}$ and $t = 1.5$ (LHS) and $t=2.25$ (RHS).}
    \label{fig:conv}
\end{figure}
Indeed, there is no reason why $q$ from~\cref{Thm:lim} should remain strictly bounded away from the obstacle. This means that it could indeed be hit after the limiting process $\eps \searrow 0$. 

\subsection{Comparison of linear, nonlinear and nonlocal dynamics}
It is interesting to observe the behavior of the solution when different dynamics are chosen rather than the nonlocal one that was considered in this article. In fact, for some $U:\R \rightarrow \R$ we can instead consider the the corresponding ``local'' problem as well, i.e., the Cauchy problem
\begin{align}
    \partial_t q_\eps  + \partial_x \big(V_\eps\big(o(x) -q_\eps\big)U_i\big(q_\eps \big)q_\eps  \big) &= 0, && (t, x) \in (0, T) \times \R, \label{eq:local} \\
    q(0, x) & = q_0(x), && x \in \R \notag
\end{align}
for $i \in \lbrace 1, 2 \rbrace$, too. The corresponding obstacle problem was intensively studied in \cite{obstacle}. In Figure~\cref{fig:comp}, we compare the solutions to ~\cref{eq:local} with initial datum $q_0^1$ and $\eps = \tfrac{1}{1024}$. In the top row, we see the solution to~\cref{eq:nocl}, in the middle row the solution to~\cref{eq:local} with $U_1(x) = \tfrac{2}{3}\big(\tfrac{3}{2}-x \big)$ for all $x \in \R$ and in the bottom row the solution to~\cref{eq:local} with $U_2(x) = 1$ (observed in~\cite{obstacle}) for all $x \in \R$.  
\begin{figure}
    \centering
    \includegraphics[clip,trim=0 17 5 5 0]{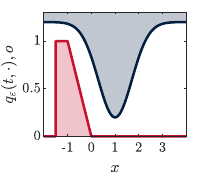}
    \includegraphics[clip,trim=20 17 5 5 0]{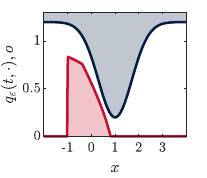}
    \includegraphics[clip,trim=20 17 5 5 0]{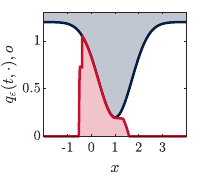}
    \includegraphics[clip,trim=20 17 0 5 0]{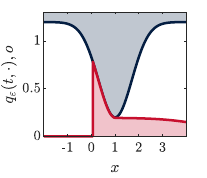} \\
    \includegraphics[clip,trim=0 17 5 5 0]{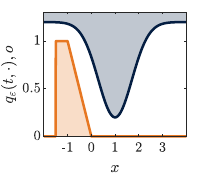}
    \includegraphics[clip,trim=20 17 5 5 0]{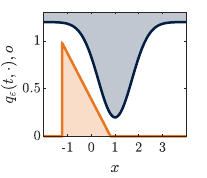}
    \includegraphics[clip,trim=20 17 5 5 0]{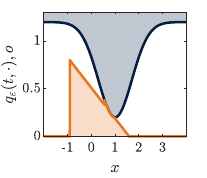}
    \includegraphics[clip,trim=20 17 0 5 0]{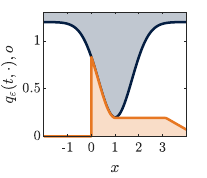} \\
    \includegraphics[clip,trim=0 0 5 5 0]{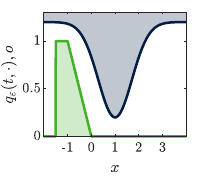}
    \includegraphics[clip,trim=20 0 5 5 0]{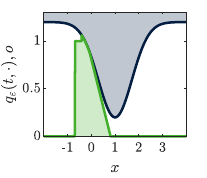}
    \includegraphics[clip,trim=20 0 5 5 0]{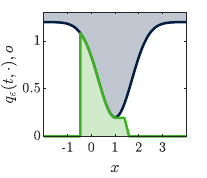}
    \includegraphics[clip,trim=20 0 0 5 0]{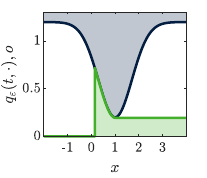}
    \caption{The solution to the nonlocal and nonlinear~\cref{eq:nocl} (top row) for $U_1(x) = \tfrac{2}{3}\big(\tfrac{3}{2}-x \big)$ ($x \in \R$), to the local and nonlinear~\cref{eq:local} for $U_1(x) = \tfrac{2}{3}\big(\tfrac{3}{2}-x \big) $ ($x \in \R$) (middle row) and the local and linear~\cref{eq:local} for $U_2(x) = 1$ ($x \in \R$) (bottom row), where $\eps = \tfrac{1}{1024}$ for times $t = 0, 0.81, 1.59, 4.5$ from left to right. Every time, the initial datum is $q_0^1$.}
    \label{fig:comp}
\end{figure}~\\
The nonlocal solution shows its characteristic smoothing effect at the top~\cite[Figure~1]{Coclite2022SL}, while the nonlinear equation develops a rarefaction wave. In the linear case, the front just moves towards the obstacle with constant speed until it is hit. In every case however, the obstacle is adhered to and a ``backward shock''~\cite[p.~4]{obstacle} can be observed as soon as the solution comes close to the obstacle. This is the case in the third picture of the top and middle row and the second picture of the bottom row.

\subsection{One sided Lipschitz bounds}
This section wants to further examine how the results in~\cref{thm:oslq,Thm:OSLV} can be visualized for $\eps = \tfrac{1}{1024}$.~\cref{fig:osl2d} provides a bird's eye view of the surface plot of $V_\eps(o-q_\eps)$ (left) and $q_\eps$ (right), when the initial datum is $q_0^1$. Here $q_\eps$ solves~\cref{eq:nocl,eq:nocl1} and $U_1(x) = \tfrac{2}{3}\big(\tfrac{3}{2} - x\big)$ for all $x \in \R$. It can be observed that the proclaimed one sided Lipschitz bounds exist for this example:
\begin{figure}
    \centering
    \includegraphics[clip,trim=0 0 7 5 0]{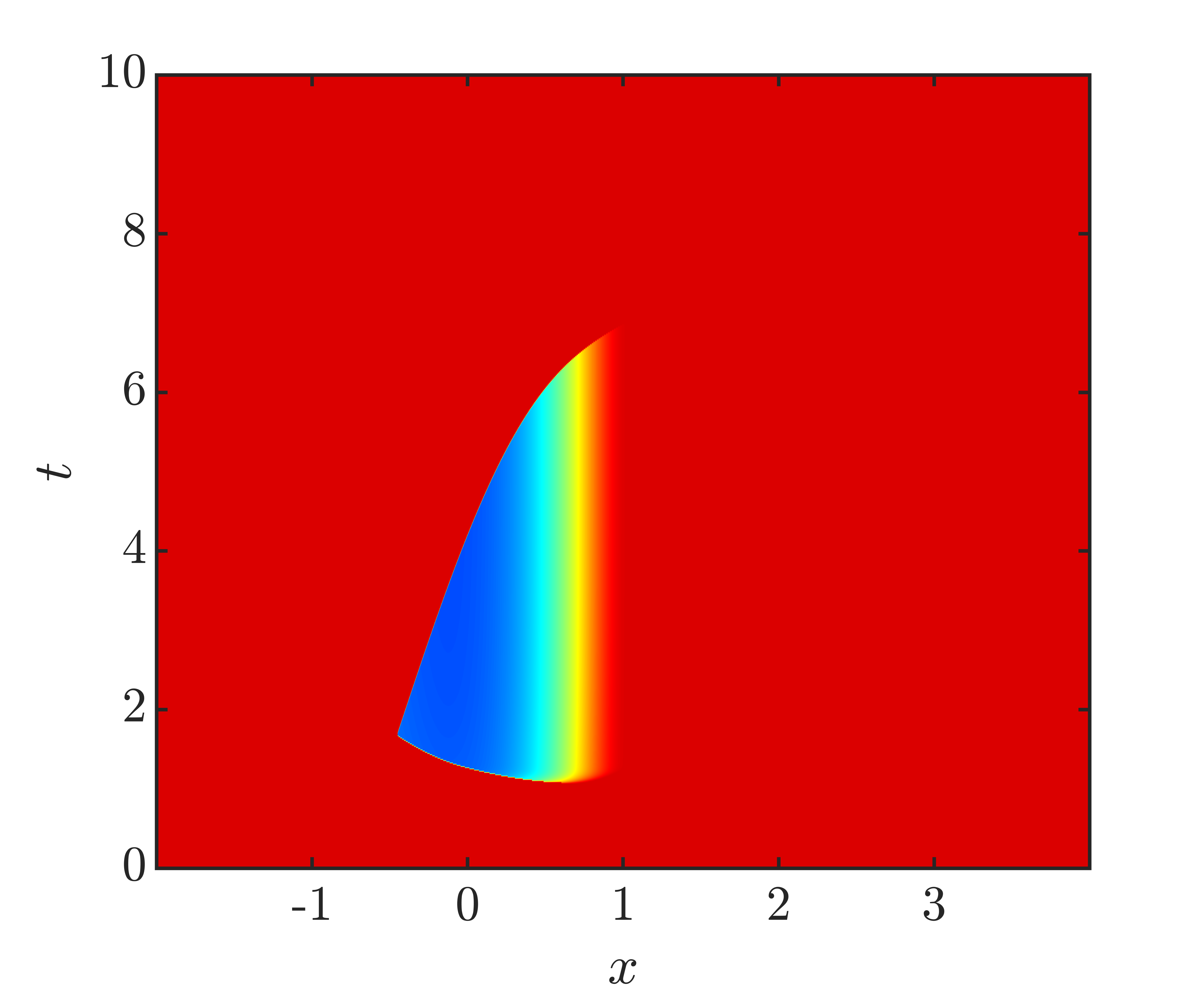}
    \includegraphics[clip,trim= 21 0 5 5 0]{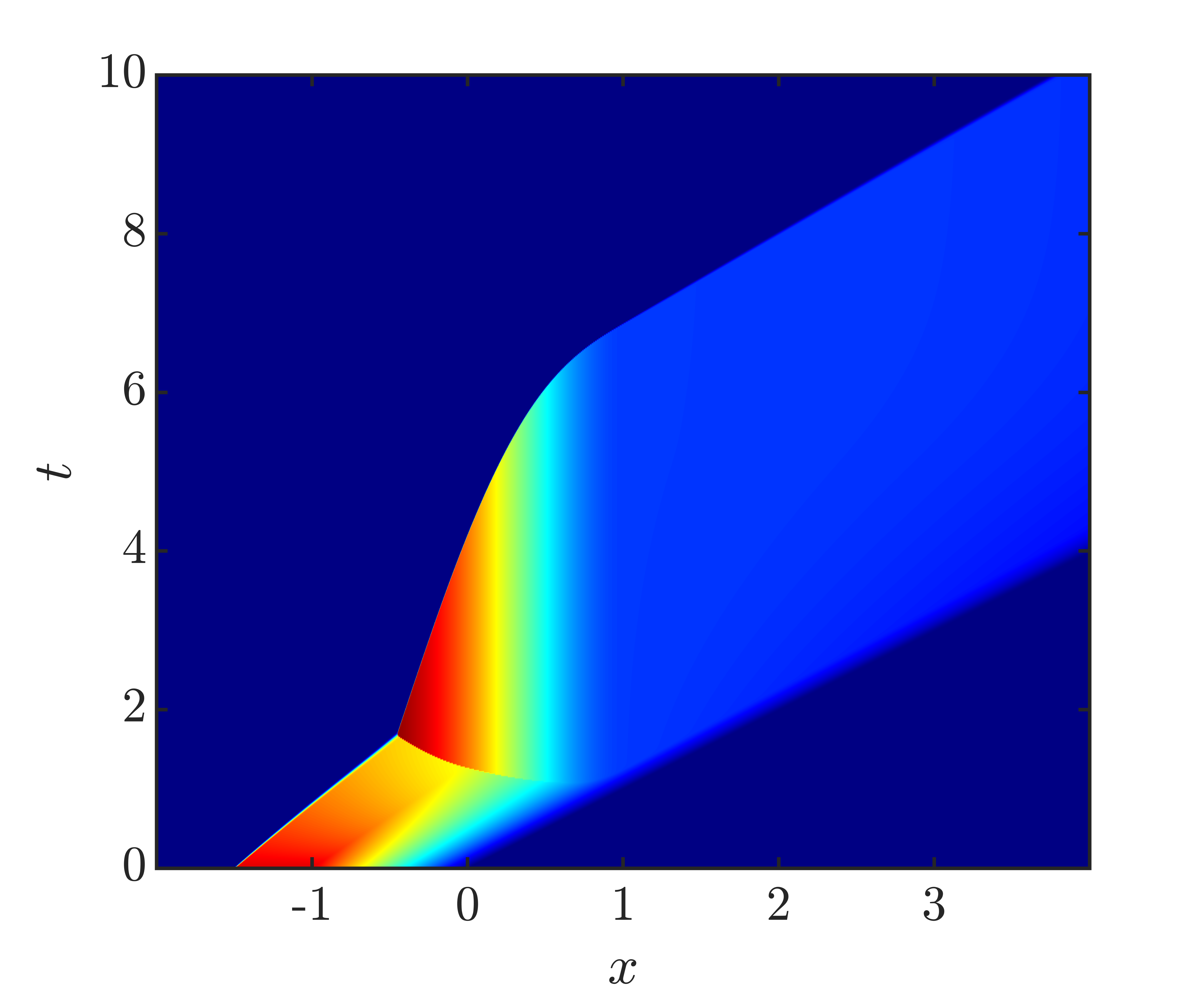}
    \caption{A surface plot of $V_\eps(o-q_\eps)$, where $q_\eps$ solves~\cref{eq:nocl,eq:nocl1} (left) for and $q_\eps$ (right) if the initial datum is $q_0^1$, the velocity field is $U_1(x) = \tfrac{2}{3}\big(\tfrac{3}{2}-x \big)$ ($x \in \R$) and $\eps = \tfrac{1}{1024}$. Colorbar: $0$\includegraphics[clip,trim=7.7 12 6 10 0]{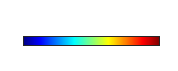}$1.1$}
    \label{fig:osl2d}
\end{figure}~\\
 For every hypothetical horizontal line intersecting with one of the two plots, the color gradient is ``smooth'' either from blue to red (in the case of $V_\eps(o-q_\eps)$) or from red to blue (when $q_\eps$) at least from one side. A similar observation also indicates that no one sided Lipschitz bounds in $t$ are to be expected for $V_\eps(o-q_\eps)$.
 
 Moreover, if the initial datum is not one sided Lipschitz bounded from below, like $q_0^2$, then,  no one sided Lipschitz develops when $t>0$ (cf.~\cref{fig:NOSL}). This can be readily derived from already existing theory on nonlocal conservation laws~\cite{keimernonlocalbalance2017}.
\begin{figure}
    \centering
    \includegraphics[clip,trim=0 0 5 5 0]{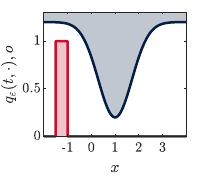}
    \includegraphics[clip,trim=20 0 5 5 0]{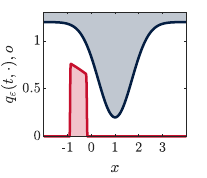}
    \includegraphics[clip,trim=20 0 5 5 0]{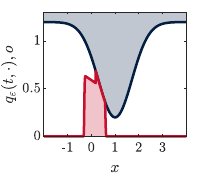}
    \includegraphics[clip,trim=20 0 5 5 0]{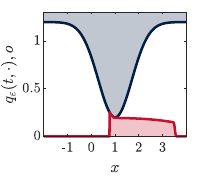}
    \caption{The solution $q_\eps$ to the nonlocal nonlinear equation~\cref{eq:nocl,eq:nocl1} for $U_1(x) = \tfrac{2}{3}\big(\tfrac{3}{2}-x \big)$ for times $t=0, 0.81, 1.59, 4.50$ (from left to right) for initial datum $q_0^2$.}
    \label{fig:NOSL}
\end{figure}

\section{Open problems}
Interesting future contributions to the topic of the obstacle problem for nonlocal scalar conservation laws could be about uniqueness of the discontinuous Cauchy problem mentioned in \cref{Thm:lim} and \cref{rem:sv}. Closely related to this is the question how the coincidence set $E(t)$ (cf.~\cref{rem:sv}) can be characterized for its regularity.

\section*{Acknowledgement}
LP and JR have been supported by the DFG -- Project-ID 416229255 -- SFB 1411. AK and JR have been supported by the DFG -- Project-ID 547096773.

\bibliographystyle{plain}
\bibliography{obstacle}

@article {2bogelein,
    AUTHOR = {B\"{o}gelein, Verena and Duzaar, Frank and Mingione, Giuseppe},
     TITLE = {Degenerate problems with irregular obstacles},
   JOURNAL = {J. Reine Angew. Math.},
  FJOURNAL = {Journal f\"{u}r die Reine und Angewandte Mathematik. [Crelle's
              Journal]},
    VOLUME = {650},
      YEAR = {2011},
     PAGES = {107--160},
      ISSN = {0075-4102,1435-5345},
   MRCLASS = {35K86 (35B65 35J87 42B37)},
  MRNUMBER = {2770559},
MRREVIEWER = {Siegfried\ Carl},
       DOI = {10.1515/CRELLE.2011.006},
       URL = {https://doi.org/10.1515/CRELLE.2011.006},
}

@incollection {1rodrigues2,
    AUTHOR = {Rodrigues, J. F.},
     TITLE = {On the hyperbolic obstacle problem of first order},
      NOTE = {Dedicated to the memory of Jacques-Louis Lions},
   JOURNAL = {Chinese Ann. Math. Ser. B},
  FJOURNAL = {Chinese Annals of Mathematics. Series B},
    VOLUME = {23},
      YEAR = {2002},
    NUMBER = {2},
     PAGES = {253--266},
      ISSN = {0252-9599,1860-6261},
   MRCLASS = {35L85 (35B35 35R35)},
  MRNUMBER = {1924141},
MRREVIEWER = {J\'{a}n\ Lov\'{\i}\v{s}ek},
       DOI = {10.1142/S0252959902000249},
}

@article {1rodrigues,
    AUTHOR = {Rodrigues, Jos\'{e} Francisco},
     TITLE = {On hyperbolic variational inequalities of first order and some
              applications},
   JOURNAL = {Monatsh. Math.},
  FJOURNAL = {Monatshefte f\"{u}r Mathematik},
    VOLUME = {142},
      YEAR = {2004},
    NUMBER = {1-2},
     PAGES = {157--177},
      ISSN = {0026-9255,1436-5081},
   MRCLASS = {35L85 (35F30 35R35 49J40)},
  MRNUMBER = {2065027},
MRREVIEWER = {Weimin\ Han},
       DOI = {10.1007/s00605-004-0238-3},
       URL = {https://doi.org/10.1007/s00605-004-0238-3},
}

@article {1amorim,
    AUTHOR = {Amorim, Paulo and Neves, Wladimir and Rodrigues, Jos\'{e}
              Francisco},
     TITLE = {The obstacle-mass constraint problem for hyperbolic
              conservation laws. {S}olvability},
   JOURNAL = {Ann. Inst. H. Poincar\'{e} C Anal. Non Lin\'{e}aire},
  FJOURNAL = {Annales de l'Institut Henri Poincar\'{e} C. Analyse Non
              Lin\'{e}aire},
    VOLUME = {34},
      YEAR = {2017},
    NUMBER = {1},
     PAGES = {221--248},
      ISSN = {0294-1449,1873-1430},
   MRCLASS = {35L65 (35R35)},
  MRNUMBER = {3592685},
MRREVIEWER = {Alberto\ Valli},
       DOI = {10.1016/j.anihpc.2015.11.003},
       URL = {https://doi.org/10.1016/j.anihpc.2015.11.003},
}

@article {1levi,
    AUTHOR = {Levi, Laurent},
     TITLE = {Obstacle problems for scalar conservation laws},
   JOURNAL = {M2AN Math. Model. Numer. Anal.},
  FJOURNAL = {M2AN. Mathematical Modelling and Numerical Analysis},
    VOLUME = {35},
      YEAR = {2001},
    NUMBER = {3},
     PAGES = {575--593},
      ISSN = {0764-583X,1290-3841},
   MRCLASS = {35L65 (35L85 35R35)},
  MRNUMBER = {1837085},
MRREVIEWER = {J.\ Chrastina},
       DOI = {10.1051/m2an:2001127},
       URL = {https://doi.org/10.1051/m2an:2001127},
}

@article {1colombo,
    AUTHOR = {Colombo, Rinaldo M. and Goatin, Paola},
     TITLE = {A well posed conservation law with a variable unilateral
              constraint},
   JOURNAL = {J. Differential Equations},
  FJOURNAL = {Journal of Differential Equations},
    VOLUME = {234},
      YEAR = {2007},
    NUMBER = {2},
     PAGES = {654--675},
      ISSN = {0022-0396,1090-2732},
   MRCLASS = {35L65 (35B30 35L45 90B20)},
  MRNUMBER = {2300671},
MRREVIEWER = {Michael\ Herty},
       DOI = {10.1016/j.jde.2006.10.014},
       URL = {https://doi.org/10.1016/j.jde.2006.10.014},
}

@article {1andreianov2,
    AUTHOR = {Andreianov, Boris and Donadello, Carlotta and Rosini,
              Massimiliano Daniele},
     TITLE = {A second-order model for vehicular traffics with local point
              constraints on the flow},
   JOURNAL = {Math. Models Methods Appl. Sci.},
  FJOURNAL = {Mathematical Models and Methods in Applied Sciences},
    VOLUME = {26},
      YEAR = {2016},
    NUMBER = {4},
     PAGES = {751--802},
      ISSN = {0218-2025,1793-6314},
   MRCLASS = {35L65 (90B20)},
  MRNUMBER = {3460622},
       DOI = {10.1142/S0218202516500172},
       URL = {https://doi.org/10.1142/S0218202516500172},
}

@article {1garavello2,
    AUTHOR = {Garavello, Mauro and Villa, Stefano},
     TITLE = {The {C}auchy problem for the {A}w-{R}ascle-{Z}hang traffic
              model with locally constrained flow},
   JOURNAL = {J. Hyperbolic Differ. Equ.},
  FJOURNAL = {Journal of Hyperbolic Differential Equations},
    VOLUME = {14},
      YEAR = {2017},
    NUMBER = {3},
     PAGES = {393--414},
      ISSN = {0219-8916,1793-6993},
   MRCLASS = {35L65 (35L45 90B20)},
  MRNUMBER = {3703014},
MRREVIEWER = {Andrea\ Tosin},
       DOI = {10.1142/S0219891617500138},
       URL = {https://doi.org/10.1142/S0219891617500138},
}

@article {1garavello,
    AUTHOR = {Garavello, M. and Goatin, P.},
     TITLE = {The {A}w-{R}ascle traffic model with locally constrained flow},
   JOURNAL = {J. Math. Anal. Appl.},
  FJOURNAL = {Journal of Mathematical Analysis and Applications},
    VOLUME = {378},
      YEAR = {2011},
    NUMBER = {2},
     PAGES = {634--648},
      ISSN = {0022-247X,1096-0813},
   MRCLASS = {35L65 (35L45 90B20)},
  MRNUMBER = {2773272},
MRREVIEWER = {Michael\ Herty},
       DOI = {10.1016/j.jmaa.2011.01.033},
       URL = {https://doi.org/10.1016/j.jmaa.2011.01.033},
}

@article {1dymski,
    AUTHOR = {Dymski, Nikodem S. and Goatin, Paola and Rosini, Massimiliano
              D.},
     TITLE = {Existence of {${BV}$} solutions for a non-conservative
              constrained {A}w-{R}ascle-{Z}hang model for vehicular traffic},
   JOURNAL = {J. Math. Anal. Appl.},
  FJOURNAL = {Journal of Mathematical Analysis and Applications},
    VOLUME = {467},
      YEAR = {2018},
    NUMBER = {1},
     PAGES = {45--66},
      ISSN = {0022-247X,1096-0813},
   MRCLASS = {35L65 (35L50 90B20)},
  MRNUMBER = {3834794},
       DOI = {10.1016/j.jmaa.2018.07.025},
       URL = {https://doi.org/10.1016/j.jmaa.2018.07.025},
}

@book{evans,
  title = {Measure Theory and Fine Properties of Functions,  Revised Edition},
  ISBN = {9781482242393},
  DOI = {10.1201/b18333},
  publisher = {Chapman and Hall/CRC},
  address = {Boca Raton},
  author = {Evans,  Lawrence Craig and Gariepy,  Ronald F.},
  year = {2015},
  month = apr 
}

@book{godlewski1991,
  title={Hyperbolic Systems of Conservation Laws},
  author={Godlewski, E. and Raviart, P.A.},
  lccn={73180239},
  series={Math{\'e}matiques \& applications},
  url={https://books.google.de/books?id=X3qyvAEACAAJ},
  year={1991},
  publisher={Ellipses},
  address = {Paris},
}

@book{evanspartial,
  title={Partial Differential Equations},
  author={Evans, L.C.},
  isbn={9780821849743},
  lccn={2009044716},
  series={Graduate studies in mathematics},
  url={https://books.google.de/books?id=Xnu0o\_EJrCQC},
  year={2010},
  publisher={American Mathematical Society}
}

@article{Simon1986CompactSI,
  title = {Compact sets in the space \({L}^{p}(0,{T}; {B})\)},
  volume = {146},
  ISSN = {1618-1891},
  url = {http://dx.doi.org/10.1007/BF01762360},
  DOI = {10.1007/bf01762360},
  number = {1},
  journal = {Annali di Matematica Pura ed Applicata},
  publisher = {Springer Science and Business Media LLC},
  author = {Simon,  Jacques},
  year = {1986},
  month = Dec,
  pages = {65–96}
}

@article{Kruzkov1970,
doi = {10.1070/SM1970v010n02ABEH002156},
url = {https://dx.doi.org/10.1070/SM1970v010n02ABEH002156},
year = {1970},
month = {feb},
publisher = {},
volume = {10},
number = {2},
pages = {217},
author = {S N Kružkov},
title = {FIRST ORDER QUASILINEAR EQUATIONS IN SEVERAL INDEPENDENT VARIABLES},
journal = {Mathematics of the USSR-Sbornik}}

@article{keimernonlocalbalance2017,
title = {Existence, uniqueness and regularity results on nonlocal balance laws},
journal = {Journal of Differential Equations},
volume = {263},
number = {7},
pages = {4023-4069},
year = {2017},
issn = {0022-0396},
doi = {https://doi.org/10.1016/j.jde.2017.05.015},
url = {https://www.sciencedirect.com/science/article/pii/S0022039617302759},
author = {Alexander Keimer and Lukas Pflug}
}

@book{alt2016eng,
  title = {Linear Functional Analysis},
  ISBN = {9781447172802},
  ISSN = {2191-6675},
  DOI = {10.1007/978-1-4471-7280-2},
  journal = {Universitext},
  publisher = {Springer},
  address = {London},
  author = {Alt,  Hans Wilhelm},
  year = {2016}
}

@book{brezis2010functional,
  title = {Functional Analysis,  Sobolev Spaces and Partial Differential Equations},
  ISBN = {9780387709147},
  DOI = {10.1007/978-0-387-70914-7},
  publisher = {Springer},
  address = {New York},
  author = {Brezis,  Haim},
  year = {2010}
}

@article{Zhang2003,
  title = {Hyperbolic conservation laws with space-dependent flux: I. Characteristics theory and Riemann problem},
  volume = {156},
  ISSN = {0377-0427},
  url = {http://dx.doi.org/10.1016/S0377-0427(02)00880-4},
  DOI = {10.1016/s0377-0427(02)00880-4},
  number = {1},
  journal = {Journal of Computational and Applied Mathematics},
  publisher = {Elsevier BV},
  author = {Zhang,  Peng and Liu,  Ru-Xun},
  year = {2003},
  month = jul,
  pages = {1–21}
}

@book{teschlbook,
  title = {Ordinary Differential Equations and Dynamical Systems},
  ISBN = {9780821891049},
  ISSN = {1065-7339},
  url = {http://dx.doi.org/10.1090/gsm/140},

  journal = {Graduate Studies in Mathematics},
  publisher = {American Mathematical Society},
  address = {Providence},
  author = {Teschl,  Gerald},
  year = {2012},
  month = aug 
}

@article{Bayen2022,
  title = {Modeling Multilane Traffic with Moving Obstacles by Nonlocal Balance Laws},
  volume = {21},
  ISSN = {1536-0040},
  url = {http://dx.doi.org/10.1137/20M1366654},
  DOI = {10.1137/20m1366654},
  number = {2},
  journal = {SIAM Journal on Applied Dynamical Systems},
  publisher = {Society for Industrial & Applied Mathematics (SIAM)},
  author = {Bayen,  Alexandre and Friedrich,  Jan and Keimer,  Alexander and Pflug,  Lukas and Veeravalli,  Tanya},
  year = {2022},
  month = jun,
  pages = {1495–1538}
}

@article{FernandezReal2020,
  title = {On the obstacle problem for the 1D wave equation},
  volume = {2},
  ISSN = {2640-3501},
  url = {http://dx.doi.org/10.3934/mine.2020026},
  DOI = {10.3934/mine.2020026},
  number = {4},
  journal = {Mathematics in Engineering},
  publisher = {American Institute of Mathematical Sciences (AIMS)},
  author = {Fernández-Real,  Xavier and Figalli,  Alessio},
  year = {2020},
  pages = {584–597}
}

@article{Milgrom2002,
  title = {Envelope Theorems for Arbitrary Choice Sets},
  volume = {70},
  ISSN = {1468-0262},
  url = {http://dx.doi.org/10.1111/1468-0262.00296},
  DOI = {10.1111/1468-0262.00296},
  number = {2},
  journal = {Econometrica},
  publisher = {The Econometric Society},
  author = {Milgrom,  Paul and Segal,  Ilya},
  year = {2002},
  month = mar,
  pages = {583–601}
}

@article{obstacle,
  doi = {10.48550/ARXIV.2405.07829},
  url = {https://arxiv.org/abs/2405.07829},
  author = {Amorim,  Paulo and Keimer,  Alexander and Pflug,  Lukas and Rodestock,  Jakob},
  title = {The obstacle problem for linear scalar conservation laws with constant velocity},
  publisher = {arXiv},
  year = {2024},
  copyright = {Creative Commons Attribution 4.0 International}
}

@article{Chiarello2018,
  title = {Global entropy weak solutions for general non-local traffic flow models with anisotropic kernel},
  volume = {52},
  ISSN = {1290-3841},
  url = {http://dx.doi.org/10.1051/m2an/2017066},
  DOI = {10.1051/m2an/2017066},
  number = {1},
  journal = {ESAIM: Mathematical Modelling and Numerical Analysis},
  publisher = {EDP Sciences},
  author = {Chiarello,  Felisia Angela and Goatin,  Paola},
  year = {2018},
  month = jan,
  pages = {163–180}
}

@book{rudin1976,
  title = {Principles of Mathematical Analysis},
  publisher = {McGraw-Hill},
  author = {Rudin, Walter},
  year = {1976},
  edition = {3}, 
  address = {New York},
}

@book{Ambrosio2000,
  title = {Functions of Bounded Variation},
  ISBN = {9781383020311},
  DOI = {10.1093/oso/9780198502456.003.0003},
  booktitle = {Functions of Bounded Variation and Free Discontinuity Problems},
  publisher = {Oxford University},
  author = {Ambrosio,  Luigi and Fusco,  Nicola and Pallara,  Diego},
  address = {Oxford},
  year = {2000},
}

@article{Amorim2015nonlocalnumerical,
  title = {On the Numerical Integration of Scalar Nonlocal Conservation Laws},
  volume = {49},
  ISSN = {1290-3841},
  url = {http://dx.doi.org/10.1051/m2an/2014023},
  DOI = {10.1051/m2an/2014023},
  number = {1},
  journal = {ESAIM: Mathematical Modelling and Numerical Analysis},
  publisher = {EDP Sciences},
  author = {Amorim,  Paulo and Colombo,  Rinaldo M. and Teixeira,  Andreia},
  year = {2015},
  month = jan,
  pages = {19–37}
}

@article{Keimer2018bounded,
  title = {Nonlocal Scalar Conservation Laws on Bounded Domains and Applications in Traffic Flow},
  volume = {50},
  ISSN = {1095-7154},
  url = {http://dx.doi.org/10.1137/18M119817X},
  DOI = {10.1137/18m119817x},
  number = {6},
  journal = {SIAM Journal on Mathematical Analysis},
  publisher = {Society for Industrial & Applied Mathematics (SIAM)},
  author = {Keimer,  Alexander and Pflug,  Lukas and Spinola,  Michele},
  year = {2018},
  month = jan,
  pages = {6271–6306}
}

@article{Chiarello2019stability,
  title = {Stability estimates for non-local scalar conservation laws},
  volume = {45},
  ISSN = {1468-1218},
  url = {http://dx.doi.org/10.1016/j.nonrwa.2018.07.027},
  DOI = {10.1016/j.nonrwa.2018.07.027},
  journal = {Nonlinear Analysis: Real World Applications},
  publisher = {Elsevier BV},
  author = {Chiarello,  Felisia Angela and Goatin,  Paola and Rossi,  Elena},
  year = {2019},
  month = feb,
  pages = {668–687}
}

@article{Santambrogio2018,
  title = {Crowd motion and evolution {PDEs} under density constraints},
  volume = {64},
  ISSN = {2267-3059},
  url = {http://dx.doi.org/10.1051/proc/201864137},
  DOI = {10.1051/proc/201864137},
  journal = {ESAIM: Proceedings and Surveys},
  publisher = {EDP Sciences},
  author = {Santambrogio,  Filippo},
  editor = {Carassus,  Laurence and Darbas,  Marion and Gayraud,  Ghislaine and Goubet,  Olivier and Salmon,  Stéphanie},
  year = {2018},
  pages = {137–157}
}

@article{Friedrich2018,
  title = {A {G}odunov type scheme for a class of {LWR} traffic flow models with non-local flux},
  volume = {13},
  ISSN = {1556-181X},
  DOI = {10.3934/nhm.2018024},
  number = {4},
  journal = {Networks and Heterogeneous Media},
  publisher = {American Institute of Mathematical Sciences (AIMS)},
  author = {Friedrich,  Jan and Kolb,  Oliver and G\"{o}ttlich,  Simone},
  year = {2018},
  pages = {531–547}
}

@article{Coclite2022BV,
  title = {On existence and uniqueness of weak solutions to nonlocal conservation laws with BV kernels},
  volume = {73},
  ISSN = {1420-9039},
  url = {http://dx.doi.org/10.1007/s00033-022-01766-0},
  DOI = {10.1007/s00033-022-01766-0},
  number = {6},
  journal = {Zeitschrift f\"{u}r angewandte Mathematik und Physik},
  publisher = {Springer Science and Business Media LLC},
  author = {Coclite,  Giuseppe Maria and De Nitti,  Nicola and Keimer,  Alexander and Pflug,  Lukas},
  year = {2022},
  month = oct 
}

@article{Coclite2022SL,
  title = {A general result on the approximation of local conservation laws by nonlocal conservation laws: The singular limit problem for exponential kernels},
  volume = {40},
  ISSN = {1873-1430},
  url = {http://dx.doi.org/10.4171/AIHPC/58},
  DOI = {10.4171/aihpc/58},
  number = {5},
  journal = {Annales de l’Institut Henri Poincaré C,  Analyse non linéaire},
  publisher = {European Mathematical Society - EMS - Publishing House GmbH},
  author = {Coclite,  Giuseppe Maria and Coron,  Jean-Michel and De Nitti,  Nicola and Keimer,  Alexander and Pflug,  Lukas},
  year = {2022},
  month = nov,
  pages = {1205–1223}
}

@article{Greenberg1959,
  title = {An Analysis of Traffic Flow},
  volume = {7},
  ISSN = {1526-5463},
  url = {http://dx.doi.org/10.1287/opre.7.1.79},
  DOI = {10.1287/opre.7.1.79},
  number = {1},
  journal = {Operations Research},
  publisher = {Institute for Operations Research and the Management Sciences (INFORMS)},
  author = {Greenberg,  Harold},
  year = {1959},
  month = feb,
  pages = {79–85}
}

@article{Lions1994,
  title = {Kinetic formulation of the isentropic gas dynamics and $p$-systems},
  volume = {163},
  ISSN = {1432-0916},
  url = {http://dx.doi.org/10.1007/BF02102014},
  DOI = {10.1007/bf02102014},
  number = {2},
  journal = {Communications in Mathematical Physics},
  publisher = {Springer Science and Business Media LLC},
  author = {Lions,  P. L. and Perthame,  B. and Tadmor,  E.},
  year = {1994},
  month = jul,
  pages = {415–431}
}

@article{Berthelin2003,
  title = {Weak solutions for a hyperbolic system with unilateral constraint and mass loss},
  volume = {20},
  ISSN = {1873-1430},
  url = {http://dx.doi.org/10.1016/S0294-1449(03)00012-X},
  DOI = {10.1016/s0294-1449(03)00012-x},
  number = {6},
  journal = {Annales de l’Institut Henri Poincaré C,  Analyse non linéaire},
  publisher = {European Mathematical Society - EMS - Publishing House GmbH},
  author = {Berthelin,  F and Bouchut,  F},
  year = {2003},
  month = dec,
  pages = {975–997}
}

@article{Berthelin2003split,
  title = {Numerical flux-splitting for a class of hyperbolic systems with unilateral constraint},
  volume = {37},
  ISSN = {1290-3841},
  url = {http://dx.doi.org/10.1051/m2an:2003038},
  DOI = {10.1051/m2an:2003038},
  number = {3},
  journal = {ESAIM: Mathematical Modelling and Numerical Analysis},
  publisher = {EDP Sciences},
  author = {Berthelin,  Florent},
  year = {2003},
  month = may,
  pages = {479–494}
}

@article{oleinik,
author={O.A.~Oleinik},
title={Discontinuous solutions of non-linear differential equations},
journal={ Uspekhi Mat. Nauk},
year={1957},
volume={12},
issue={3(75)},
pages={ 3--73},
doi={http://mi.mathnet.ru/eng/rm7611}
}

@article{Friedrich2024,
  title = {Conservation Laws with Nonlocal Velocity: The Singular Limit Problem},
  volume = {84},
  ISSN = {1095-712X},
  url = {http://dx.doi.org/10.1137/22M1530471},
  DOI = {10.1137/22m1530471},
  number = {2},
  journal = {SIAM Journal on Applied Mathematics},
  publisher = {Society for Industrial & Applied Mathematics (SIAM)},
  author = {Friedrich,  Jan and G\"{o}ttlich,  Simone and Keimer,  Alexander and Pflug,  Lukas},
  year = {2024},
  month = mar,
  pages = {497–522}
}

@misc{colombo2023overview,
  doi = {10.48550/ARXIV.2311.14528},
  author = {Colombo,  Maria and Crippa,  Gianluca and Marconi,  Elio and Spinolo,  Laura V.},
  title = {An overview on the local limit of non-local conservation laws,  and a new proof of a compactness estimate},
  publisher = {arXiv},
  year = {2023},
  copyright = {Creative Commons Attribution 4.0 International}
}

@article{Colombo2019,
  title = {On the Singular Local Limit for Conservation Laws with Nonlocal Fluxes},
  volume = {233},
  ISSN = {1432-0673},
  url = {http://dx.doi.org/10.1007/s00205-019-01375-8},
  DOI = {10.1007/s00205-019-01375-8},
  number = {3},
  journal = {Archive for Rational Mechanics and Analysis},
  publisher = {Springer Science and Business Media LLC},
  author = {Colombo,  Maria and Crippa,  Gianluca and Spinolo,  Laura V.},
  year = {2019},
  month = mar,
  pages = {1131–1167}
}

@article{Crippa2021,
  title = {Local limit of nonlocal traffic models: Convergence results and total variation blow-up},
  volume = {38},
  ISSN = {1873-1430},
  url = {http://dx.doi.org/10.1016/j.anihpc.2020.12.002},
  DOI = {10.1016/j.anihpc.2020.12.002},
  number = {5},
  journal = {Annales de l’Institut Henri Poincaré C,  Analyse non linéaire},
  publisher = {European Mathematical Society - EMS - Publishing House GmbH},
  author = {Crippa,  Gianluca and Marconi,  Elio and Spinolo,  Laura V. and Colombo,  Maria},
  year = {2021},
  month = oct,
  pages = {1653–1666}
}

@article{Keimer2023disc,
  title = {On the singular limit problem for a discontinuous nonlocal conservation law},
  volume = {237},
  ISSN = {0362-546X},
  url = {http://dx.doi.org/10.1016/j.na.2023.113381},
  DOI = {10.1016/j.na.2023.113381},
  journal = {Nonlinear Analysis},
  publisher = {Elsevier BV},
  author = {Keimer,  Alexander and Pflug,  Lukas},
  year = {2023},
  month = dec,
  pages = {113381}
}

@article{Colombo2023,
  title = {Nonlocal Traffic Models with General Kernels: Singular Limit,  Entropy Admissibility,  and Convergence Rate},
  volume = {247},
  ISSN = {1432-0673},
  url = {http://dx.doi.org/10.1007/s00205-023-01845-0},
  DOI = {10.1007/s00205-023-01845-0},
  number = {2},
  journal = {Archive for Rational Mechanics and Analysis},
  publisher = {Springer Science and Business Media LLC},
  author = {Colombo,  Maria and Crippa,  Gianluca and Marconi,  Elio and Spinolo,  Laura V.},
  year = {2023},
  month = feb 
}

@article{Keimer2025,
  title = {On the singular limit problem for nonlocal conservation laws: A general approximation result for kernels with fixed support},
  volume = {547},
  ISSN = {0022-247X},
  url = {http://dx.doi.org/10.1016/j.jmaa.2025.129307},
  DOI = {10.1016/j.jmaa.2025.129307},
  number = {2},
  journal = {Journal of Mathematical Analysis and Applications},
  publisher = {Elsevier BV},
  author = {Keimer,  Alexander and Pflug,  Lukas},
  year = {2025},
  month = jul,
  pages = {129307}
}

@misc{dn25,
  doi = {10.48550/ARXIV.2511.15631},
  author = {Coclite,  Giuseppe Maria and De Nitti,  Nicola and Huang,  Kuang},
  title = {Singular limit for a class of nonlocal conservation laws via compensated compactness},
  publisher = {arXiv},
  year = {2025},
  copyright = {Creative Commons Attribution 4.0 International}
}

@article{Dumitrescu2014,
  title = {Optimal Stopping for Dynamic Risk Measures with Jumps and Obstacle Problems},
  volume = {167},
  ISSN = {1573-2878},
  url = {http://dx.doi.org/10.1007/s10957-014-0635-2},
  DOI = {10.1007/s10957-014-0635-2},
  number = {1},
  journal = {Journal of Optimization Theory and Applications},
  publisher = {Springer Science and Business Media LLC},
  author = {Dumitrescu,  Roxana and Quenez,  Marie-Claire and Sulem,  Agnès},
  year = {2014},
  month = aug,
  pages = {219–242}
}

@article{Rodriguez-Aros2016-em,
  title = {Mathematical justification of an elastic elliptic membrane obstacle problem},
  volume = {345},
  ISSN = {1873-7234},
  url = {http://dx.doi.org/10.1016/j.crme.2016.10.014},
  DOI = {10.1016/j.crme.2016.10.014},
  number = {2},
  journal = {Comptes Rendus. Mécanique},
  publisher = {MathDoc/Centre Mersenne},
  author = {Rodriguez-Arós, Angel},
  year = {2016},
  month = Nov,
  pages = {153–157}
}

@article{Perninge2025,
  title = {Optimal stopping of BSDEs with constrained jumps and related double obstacle PDEs},
  volume = {32},
  ISSN = {1420-9004},
  url = {http://dx.doi.org/10.1007/s00030-025-01026-w},
  DOI = {10.1007/s00030-025-01026-w},
  number = {2},
  journal = {Nonlinear Differential Equations and Applications NoDEA},
  publisher = {Springer Science and Business Media LLC},
  author = {Perninge,  Magnus},
  year = {2025},
  month = jan 
}

@article{Andreianov2016,
  title = {A second-order model for vehicular traffics with local point constraints on the flow},
  volume = {26},
  ISSN = {1793-6314},
  url = {http://dx.doi.org/10.1142/S0218202516500172},
  DOI = {10.1142/s0218202516500172},
  number = {04},
  journal = {Mathematical Models and Methods in Applied Sciences},
  publisher = {World Scientific Pub Co Pte Ltd},
  author = {Andreianov,  Boris and Donadello,  Carlotta and Rosini,  Massimiliano Daniele},
  year = {2016},
  month = feb,
  pages = {751–802}
}

@article{Benyahia2017,
  title = {A macroscopic traffic model with phase transitions and local point constraints on the flow},
  volume = {12},
  ISSN = {1556-181X},
  url = {http://dx.doi.org/10.3934/nhm.2017013},
  DOI = {10.3934/nhm.2017013},
  number = {2},
  journal = {Networks \&amp; Heterogeneous Media},
  publisher = {American Institute of Mathematical Sciences (AIMS)},
  author = {Benyahia,  Mohamed and D. Rosini,  Massimiliano},
  year = {2017},
  pages = {297–317}
}

@article{Marcellini2020,
  title = {The {Riemann} problem for a Two-Phase model for road traffic with fixed or moving constraints},
  volume = {17},
  ISSN = {1551-0018},
  url = {http://dx.doi.org/10.3934/mbe.2020062},
  DOI = {10.3934/mbe.2020062},
  number = {2},
  journal = {Mathematical Biosciences and Engineering},
  publisher = {American Institute of Mathematical Sciences (AIMS)},
  author = {Marcellini,  Francesca},
  year = {2020},
  pages = {1218–1232}
}

@article{Lions1967,
  title = {Variational inequalities},
  volume = {20},
  ISSN = {1097-0312},
  url = {http://dx.doi.org/10.1002/cpa.3160200302},
  DOI = {10.1002/cpa.3160200302},
  number = {3},
  journal = {Communications on Pure and Applied Mathematics},
  publisher = {Wiley},
  author = {Lions,  J. L. and Stampacchia,  G.},
  year = {1967},
  month = aug,
  pages = {493–519}
}

@article{Caffarelli1977,
  title = {The regularity of free boundaries in higher dimensions},
  volume = {139},
  ISSN = {0001-5962},
  url = {http://dx.doi.org/10.1007/BF02392236},
  DOI = {10.1007/bf02392236},
  number = {0},
  journal = {Acta Mathematica},
  publisher = {International Press of Boston},
  author = {Caffarelli,  Luis A.},
  year = {1977},
  pages = {155–184}
}

@misc{keimerliverani26,
  doi = {10.48550/ARXIV.2604.08904},
  author = {Gong,  Xiaoqian and Keimer,  Alexander and Liverani,  Lorenzo and Matin,  Hossein Nick Zinat},
  title = {Existence and uniqueness of nonlocal nonlinear conservation laws via fixed-point methods},
  publisher = {arXiv},
  year = {2026},
  copyright = {Creative Commons Attribution Non Commercial No Derivatives 4.0 International}
}

@misc{decourcel2025,
  doi = {10.48550/ARXIV.2512.13535},
  author = {de Courcel,  Antonin Chodron},
  title = {Well-posedness of multidimensional nonlocal conservation laws with nonlinear mobility and bounded force},
  publisher = {arXiv},
  year = {2025},
  copyright = {Creative Commons Attribution 4.0 International}
}

@article{Bulek2017,
  title = {On unified theory for scalar conservation laws with fluxes and sources discontinuous with respect to the unknown},
  volume = {262},
  ISSN = {0022-0396},
  url = {http://dx.doi.org/10.1016/j.jde.2016.09.020},
  DOI = {10.1016/j.jde.2016.09.020},
  number = {1},
  journal = {Journal of Differential Equations},
  publisher = {Elsevier BV},
  author = {Bulíček,  Miroslav and Gwiazda,  Piotr and Świerczewska-Gwiazda,  Agnieszka},
  year = {2017},
  month = Jan,
  pages = {313–364}
}

@article{BULEK2011,
  title = {On scalar hyperbolic conservation law with a discontinuous flux},
  volume = {21},
  ISSN = {1793-6314},
  url = {http://dx.doi.org/10.1142/S021820251100499X},
  DOI = {10.1142/s021820251100499x},
  number = {01},
  journal = {Mathematical Models and Methods in Applied Sciences},
  publisher = {World Scientific Pub Co Pte Ltd},
  author = {Bulíček,  Miroslav and Gwiazda,  Piotr and Málek,  Josef and Świerczewska-Gwiazda,  Agnieszka},
  year = {2011},
  month = Jan,
  pages = {89–113}
}

@article{Kloeden2016,
  title = {Nonlocal multi-scale traffic flow models: analysis beyond vector spaces},
  volume = {6},
  ISSN = {1664-3615},
  url = {http://dx.doi.org/10.1007/s13373-016-0090-5},
  DOI = {10.1007/s13373-016-0090-5},
  number = {3},
  journal = {Bulletin of Mathematical Sciences},
  publisher = {World Scientific Pub Co Pte Ltd},
  author = {Kloeden,  Peter E. and Lorenz,  Thomas},
  year = {2016},
  month = Aug,
  pages = {453}
  }

@misc{keimer2026bothsides,
  doi = {10.48550/ARXIV.2605.00635},
  author = {Keimer,  Alexander and Pflug,  Lukas},
  title = {Nonlocal Approximation Principle for Entropy Solutions of Scalar Conservation Laws},
  publisher = {arXiv},
  year = {2026},
  copyright = {Creative Commons Attribution 4.0 International}
}
\end{document}